\documentclass[a4paper,reqno,11pt]{article}
\usepackage{calrsfs}
\DeclareMathAlphabet{\pazocal}{OMS}{zplm}{m}{n}
\usepackage{amsmath, amsfonts, amssymb, amsthm, amscd}
\usepackage{graphicx}
\usepackage{psfrag}
\usepackage{perpage}
\usepackage{url}
\usepackage{color}
\usepackage{lmodern}
\usepackage{dsfont}
\usepackage{epsf}
\usepackage{graphics}
\usepackage{epsfig}
\usepackage{epstopdf} 
\usepackage{enumerate}

\usepackage[utf8]{inputenc}
\usepackage[T1]{fontenc}
\usepackage{microtype}

\usepackage[a4paper,scale={0.72,0.74},marginratio={1:1},footskip=7mm,headsep=10mm]{geometry}

\usepackage{hyperref}

\usepackage{verbatim} 

\numberwithin{equation}{section}

\newtheorem{theorem}{Theorem}[section]
\newtheorem{lemma}[theorem]{Lemma}
\newtheorem{proposition}[theorem]{Proposition}

\newtheorem{remark}[theorem]{Remark}

\newtheorem{claim}[theorem]{Claim}

\newcommand{\bbE}{{\ensuremath{\mathbb E}} }

\newcommand{\bbN}{{\ensuremath{\mathbb N}} }

\newcommand{\bbP}{{\ensuremath{\mathbb P}} }

\newcommand{\bbR}{{\ensuremath{\mathbb R}} }

\newcommand{\bbZ}{{\ensuremath{\mathbb Z}} }

\newcommand{\cA}{{\ensuremath{\pazocal A}} }
\newcommand{\cB}{{\ensuremath{\pazocal B}} }
\newcommand{\cC}{{\ensuremath{\pazocal C}} }
\newcommand{\cD}{{\ensuremath{\pazocal D}} }
\newcommand{\cE}{{\ensuremath{\pazocal E}} }

\newcommand{\cH}{{\ensuremath{\pazocal H}} }

\newcommand{\cO}{{\ensuremath{\pazocal O}} }

\newcommand{\cR}{{\ensuremath{\pazocal R}} }
\newcommand{\cS}{{\ensuremath{\pazocal S}} }

\newcommand{\gd}{\delta}

\newcommand{\bep}{\overline{\epsilon}}
\newcommand{\gep}{\varepsilon}       

\newcommand{\gl}{\lambda}

\renewcommand{\tilde}{\widetilde}          
\DeclareMathSymbol{\leqslant}{\mathalpha}{AMSa}{"36} 
\DeclareMathSymbol{\geqslant}{\mathalpha}{AMSa}{"3E} 
\DeclareMathSymbol{\eset}{\mathalpha}{AMSb}{"3F}     
\newcommand{\dd}{\text{\rm d}}             

\newcommand{\sumtwo}[2]{\sum_{\substack{#1 \\ #2}}} 

\newcommand{\one}{{\mathchoice {1\mskip-4mu\mathrm l}
         {1\mskip-4mu\mathrm l}
         {1\mskip-4.5mu\mathrm l}
         {1\mskip-5mu\mathrm l}}}

\newcommand{\R}{\mathbb{R}}

\newcommand{\Z}{\mathbb{Z}}
\newcommand{\N}{\mathbb{N}}

\newcommand{\PEfont}{\mathrm}

\def\p{\ensuremath{\PEfont P}}
\def\e{\ensuremath{\PEfont E}}
\newcommand{\E}{\e}
\renewcommand{\P}{\p}

\renewcommand{\epsilon}{\varepsilon}
\renewcommand{\rho}{\varrho}
\renewcommand{\phi}{\varphi}

\newcommand{\cpl}[1]{{#1}^{c}}

\newcommand{\con}[3]{#1 \stackrel{#2}{\longleftrightarrow}#3}

\newenvironment{myenumerate}{%
\renewcommand{\theenumi}{\arabic{enumi}}%
\renewcommand{\labelenumi}{{\rm(\theenumi)}}%
\begin{list}{\labelenumi}
	{%
	\setlength{\itemsep}{0.4em}%
	\setlength{\topsep}{0.5em}%
	\setlength\leftmargin{2.45em}%
	\setlength\labelwidth{2.05em}%
	\setlength{\labelsep}{0.4em}%
	\usecounter{enumi}%
	}%
	}%
{\end{list}
}

\renewenvironment{enumerate}{
\begin{myenumerate}}%
{\end{myenumerate}}

\newenvironment{myitemize}{%
\begin{list}{$\bullet$}%
 	{%
	\setlength{\itemsep}{0.4em}%
	\setlength{\topsep}{0.5em}%
	\setlength\leftmargin{2.45em}%
	\setlength\labelwidth{2.05em}%
	\setlength{\labelsep}{0.4em}%
	}%
	}%
{\end{list}}

\renewenvironment{itemize}{
\begin{myitemize}}%
{\end{myitemize}}

\MakePerPage[2]{footnote} 

\def\dd{\mathrm{d}}

\newcommand{\Leb}{\text{\rm Leb}} 
\newcommand{\out}{\text{\rm out}} 
\newcommand{\ins}{\text{\rm in}} 
 
\newcommand{\ext}{\text{\rm ext}} 
\newcommand{\cross}{\text{\rm Cross}}
\newcommand{\conn}{\text{\rm conn}}

\definecolor{light-gray}{gray}{0.5}

\def\cAJ{\cA_{R,\epsilon,\bep}} 

\begin{document}


\title{Brownian Paths Homogeneously Distributed in Space: Percolation Phase Transition and Uniqueness of the Unbounded Cluster}

\author{
Dirk Erhard
\footnotemark[1]
\\
Juli\'{a}n Mart\'inez
\footnotemark[2]
\\
Julien Poisat
\footnotemark[3]
}

\footnotetext[1]{
Mathematics Institute,
Warwick University,
Coventry, CV4 7AL, UK,\\
{\sl D.Erhard@warwick.ac.uk}
}
\footnotetext[2]{
Instituto de Investigaciones Matem\'aticas Luis A. Santal\'o, Conicet,
C1428EGA, Buenos Aires, Argentina,\\
{\sl jmartine@dm.uba.ar} and {\sl martinez@math.leidenuniv.nl}
}

\footnotetext[3]{
CEREMADE, Universit\'e Paris-Dauphine, UMR CNRS 7534, Place du Mar\'echal de Lattre de Tassigny, 75775 CEDEX-16 Paris, France,\\
{\sl poisat@ceremade.dauphine.fr}
}

\date{\today}
\maketitle

\begin{abstract}
We consider a continuum percolation model on $\R^d$, $d\geq 1$.
For $t,\lambda\in (0,\infty)$ and $d\in\{1,2,3\}$, the occupied set is given by 
the union of independent Brownian paths running up to time $t$ whose
initial points form a Poisson point process with intensity $\lambda>0$.
When $d\geq 4$, the Brownian paths are replaced by Wiener sausages
with radius $r>0$.\\
We establish that, for $d=1$ and all choices of $t$, no percolation occurs,
whereas for $d\geq 2$, there is a non-trivial percolation transition
in $t$, provided $\lambda$ and $r$ are chosen properly.
The last statement means that $\lambda$ has to be chosen to be strictly smaller than the critical percolation parameter for the occupied set at time zero
(which is infinite when $d\in\{2,3\}$, but finite and dependent on $r$ when $d\geq 4$).
We further show that for all $d\geq 2$, the unbounded cluster in the supercritical phase is unique.\\
Along the way a finite box criterion for non-percolation in the Boolean model is extended to radius distributions with an exponential tail. 
This may be of independent interest.
The present paper settles the basic properties of the model and should be viewed as a jumpboard for finer results.

\medskip\noindent

{\it Acknowledgments.} DE and JP were supported by ERC Advanced Grant 267356 VARIS. JP held a postdoc position at the Mathematical Institute of Leiden University during the preparation of this paper.
JM was
supported by Erasmus Mundus scholarship BAPE-2009-1669.
The authors are grateful to R. Meester and M. Penrose for providing unpublished
notes, which already contain a sketch of the proof of Proposition \ref{lem:continuity}. They thank J.-B. Gou\'{e}r\'{e} for valuable comments on
the preliminary version as well as an anonymous referee for suggesting improvements in the presentation of the paper. JM is grateful to S. Lopez for valuable discussions.
\end{abstract}

\newpage




\section{Introduction}
\label{S1}

{\it Notation.} For every $d\geq1$, we denote by $\Leb_d$ the Lebesgue measure on $\R^d$. 
$||\cdot||$ and $||\cdot||_{\infty}$ stand for the Euclidean norm and supremum norm on $\R^d$, respectively.
For any set $A$, the symbols $\cpl{A}$ and $\overline{A}$ refer to the complement set and the closure of $A$ respectively.
The open ball with center $z$ and radius $r$ with respect to the Euclidean norm is denoted by $\cB(z,r)$, whereas $\cB_{\infty}(z,r)$ stands for the 
same ball with respect to the supremum norm. 
Furthermore, for every $0<r<r'$, we denote by $\cA(r,r') = \cB(0,r')\setminus \overline{\cB}(0,r)$ and $\cA_{\infty}(r,r') = \cB_{\infty}(0,r')\setminus
\overline{\cB}_{\infty}(0,r)$ the annulus delimited by the balls of radii $r$ and $r'$
with respect to the Euclidean norm and supremum norm, respectively. 
For all $I \subseteq \bbR^+$, we denote by $B_I$ the set $\{B_t,\, t\in I\}$.
The symbol $\bbP^a$ denotes the law of a Brownian motion starting at $a$. Finally,
$\bbP^{a_1,a_2}$ denotes the law of two independent Brownian motions starting at $a_1$ and $a_2$, respectively. 

\subsection{Overview}
\label{S1.1}
\par For $\lambda>0$, let $(\Omega_p, \pazocal{A}_p, \P_{\lambda})$ be a probability space on which a Poisson point process $\cE$ with intensity $\lambda \times
\Leb_d$ is defined. Conditionally on $\cE$, we fix a collection of independent Brownian motions $\{(B_t^x)_{t\geq 0},\, x\in \cE\}$ such that
for each $x\in\cE$, $B_0^x = x$ and $(B_t^x-x)_{t\geq0}$ is independent of $\cE$. We study for $t,r\geq0$ the {\it occupied} set (see Figure \ref{simulations} below):
\begin{equation}\label{def:occB}
\cO_{t,r} := \left\{
\begin{array}{ll}
 \bigcup_{x\in\cE} \bigcup_{0\leq s \leq t}\, \cB(B_s^x, r),\quad &\mbox{if }r>0,\\
 \bigcup_{x\in\cE} B_{[0,t]}^x,\quad &\mbox{if }r=0.
 \end{array}\right.
\end{equation}
In the rest of the paper, we write $\cO_{t}$ instead of $\cO_{t,{ 0}}$.
From now on we will denote by $\P$ the probability measure on the space where $\cO_{t,r}$ is defined, see Remark \ref{rem:ergodic}.
\begin{remark}
\label{rem:ergodic}
A more rigorous definition of the model described above can be done along similar lines as in Section 1.4 of \cite{MR96} for the Boolean percolation model. One consequence of that construction is the ergodicity of $\cO_{t,r}$ with respect to shifts in space. 
\end{remark}

\begin{figure}[h]
\begin{center}
\begin{tabular}{ccc}
\includegraphics[height=3cm]{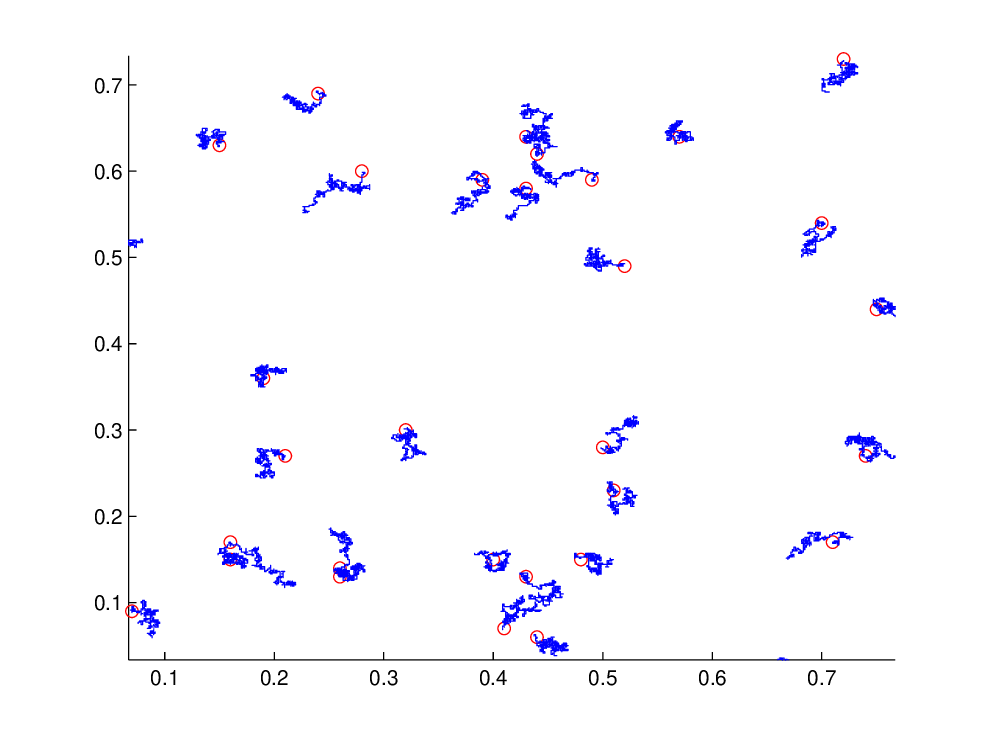} &
\includegraphics[height=3cm]{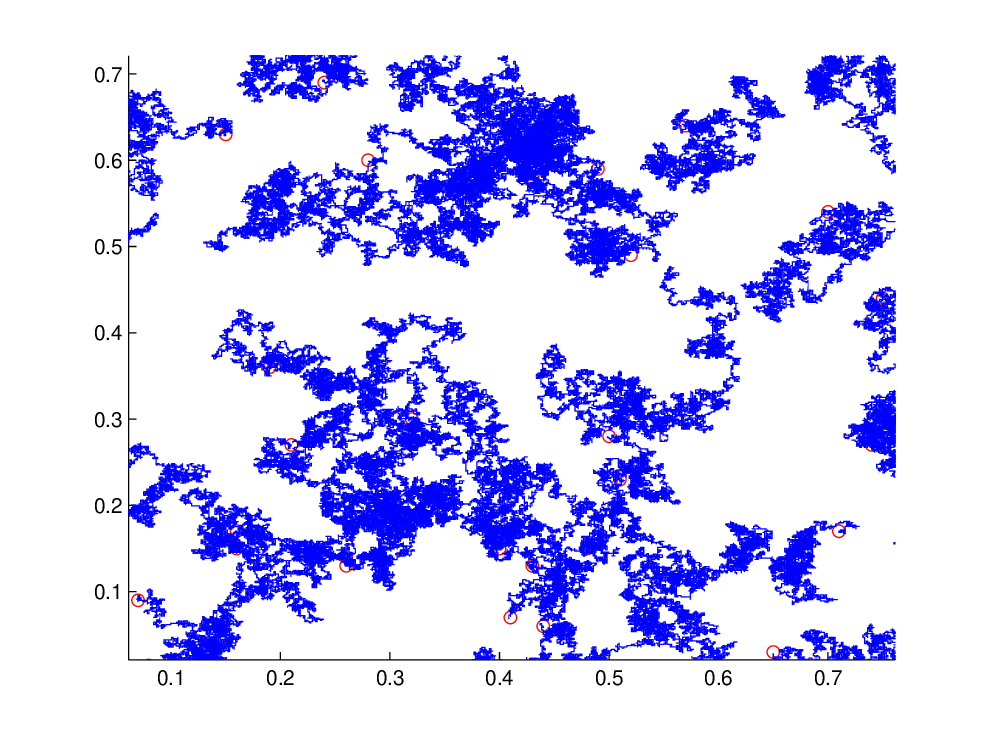} &
\includegraphics[height=3cm]{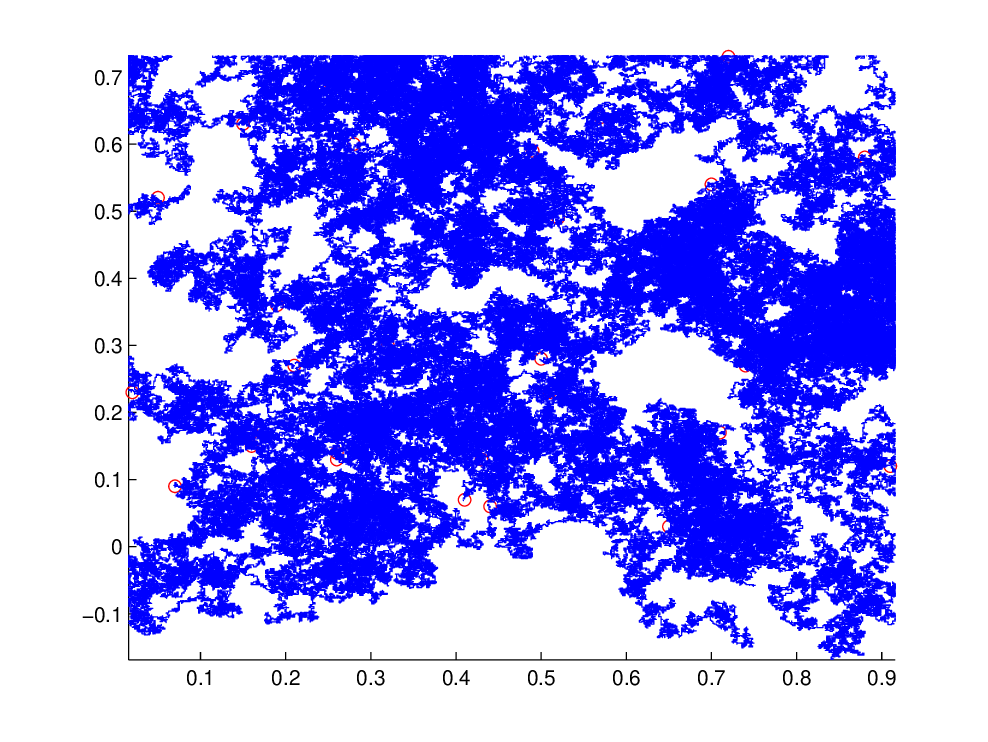}
\end{tabular}
\end{center}
\caption{\small
Simulations of $\cO_t$ in the case $d=2$, at small, intermediate, and large times.
}
\label{simulations}
\end{figure}

\par Two points $x$ and $y$ of $\bbR^d$ are said to be connected in $\cO_{t,r}$ if and only if there exists a continuous function $\gamma: [0,1]\mapsto
\cO_{t,r}$ such that $\gamma(0)=x$ and $\gamma(1)=y$. A subset of $\cO_{t,r}$ is connected if and only if all of its points are pairwise connected. In the
following a connected subset of $\cO_{t,r}$ is called a component. 
A component is bounded if it is contained in $\cB(0,R)$ for some $R>0$.
Otherwise, the component is said to be unbounded. 
A {\it cluster} is a connected component which is maximal in the sense that it is not strictly contained in another connected component.
Clusters will be denoted by $C$ all over this work. 
We say that our model percolates if $\cO_{t,r}$ contains at least one unbounded cluster.\\

\par We are interested in the percolative properties of the occupied set: is there an unbounded cluster for large $t$? Is it unique? What happens for small $t$?
Since an elementary monotonicity argument shows that $t\mapsto\cO_{t,r}$ is non-decreasing, the first and the third question may be rephrased as follows: is there a percolation transition in $t$?

\subsection{Results}
\label{S1.2}
We fix $\lambda>0$.

\begin{theorem}{\rm {\bf [No percolation for $d=1$]}}
\label{thm:noperc}
Let $d=1$. Then, for all $t\geq 0$, the set $\cO_t$ has almost surely
no unbounded cluster.
\end{theorem}

\begin{theorem}{\rm{\bf [Percolation phase transition and uniqueness for $d\in\{2,3\}$]}}
\label{thm:Aexist}
Suppose that $d\in\{2,3\}$. There exists $t_c = t_c(\lambda,d)>0$ such that for $t<t_c$, $\cO_t$ has almost surely no unbounded cluster, whereas for $t>t_c$, $\cO_t$
has almost surely a unique unbounded cluster.
\end{theorem}

\noindent
Let $d\geq 4$, $r>0$ and let $\delta_r$ be the Dirac measure concentrated
on $r$. We denote by $\lambda_c(\delta_r)$ the critical value for
$\cO_{0,r}$ such that for all $\lambda<\lambda_c(\delta_r)$ the set $\cO_{0,r}$
almost surely does not contain an unbounded cluster, and such that for
$\lambda>\lambda_c(\delta_r)$ it does, see also (\ref{eq:lambdacrit}).
It follows from Theorem \ref{thm:Gouere}, that $\lambda_c(\delta_r)>0$ and
$\lim_{r\to 0}\lambda_c(\delta_r)=\infty$.
 
\begin{theorem}{\rm{\bf [Percolation phase transition and uniqueness for $d\geq4$]}}
\label{thm:Bexist}
Suppose that $d\geq4$ and let $r>0$ be such that $\lambda<\lambda_c(\delta_r)$. Then, there exists $t_c = t_c(\lambda,d,r)>0$ such that for $t<t_c$,
$\cO_{t,r}$ has almost surely no unbounded cluster, whereas for $t>t_c$, it has almost surely a unique  unbounded cluster.
\end{theorem}

\subsection{Discussion}
\label{S1.4}

{\it Motivation and related models.} Our model fits into the class of continuum percolation models, which have been studied by both mathematicians and physicists. 
Their first appearance can be traced back (at least) to Gilbert \cite{G61} under the name of random plane networks. 
Gilbert was interested in modeling infinite communication networks of stations with range $R>0$. 
This was done by connecting any pair of points of a Poisson point process on $\R^2$ whenever their distance is less than $R$. 
Another application, which is mentioned in this work is the modeling of a contagious infection. 
Here, each individual gets infected when it has distance less than $R$ to an infected individual.

\par A subclass of continuum percolation models follows the following recipe: attach to each point of a point process (e.g. a Poisson point process) a random geometric object, e.g.\ a disk of random radius (Boolean model) or a segment of random length and random orientation (Poisson sticks model or needle percolation). 
Our model also falls into this class: we attach to each point of a Poisson point process a Brownian path (a path of a Wiener sausage when $d\geq 4$). It could actually be seen as a model of defects randomly distributed in a material that propagate at random, see Menshikov, Molchanov and Sidorenko~\cite{MMS88} for other physical motivations of continuum percolation. One can think for example
of an (infinite) piece of wood containing (homogeneously distributed) worms, where each worm tunnels through the piece of wood at random, and we wonder when the
latter ``breaks''.

\par The informal description above is reminiscent of (and actually, borrowed from) the problem of the disconnection of a cylinder by a
random
walk, which itself is linked to interlacement percolation \cite{S10}.
The latter is given by the random subset obtained when looking at the trace of a simple random walk on the torus $(\Z/N\Z)^{d}$ started from the uniform
distribution and running up to time $uN^d$, as $N\uparrow \infty$. Here $u$ plays the role of an intensity parameter for the interlacements set. However, even
though the model of random interlacements and our model seem to share some similarities, there is an important
difference: in the interlacement model, the number of trajectories which enter a ball of radius $R$ scales like $cR^{d-2}$ for some $c>0$,  whereas
in our case it is at least of order $R^d$. Nevertheless, we expect that a continuous version of random interlacement should arise as a scaling limit of our model as (i) time goes to infinity, (ii) intensity goes to $0$ and (iii) the product of both quantities stays constant.

\par For $d\geq 4$, our model actually appears in \u{C}ern\'{y}, Funken and Spodarev~\cite{CFS08} and describes the target detection area of a network of mobile sensors initially distributed at random and moving according to Brownian dynamics. However, in this work the focus is on numerical computations of coverage probabilities rather than on percolation. In a similar spirit Kesidis, Kostantopoulos and Phoha~\cite{KKP05} provide formulas for the detection time of a particle positioned at the origin (explicitly for $d=3$, bounds for $d=2$).  Percolation properties for a network of mobile sensors have also been studied by Peres, Sinclair, Sousi and Stauffer~\cite{PSSS13,PSS13}.
Nonetheless, instead of looking at $\cO_{t,r}$, which contains all paths up to time $t$ of the field of Brownian motions, they look at $\cup_{x\in\cE} \cB(B_t^x,r)$ at \emph{each fixed time $t$}.
This is an example of a dynamic Boolean model, as introduced by van den Berg, Meester and White~\cite{vdBMW97}.

\par Finally, another motivation to study such a model is that it should arise as the scaling limit of a certain class of discrete dependent percolation models; more precisely, percolation models for a system of independent finite-time random walks initially homogeneously distributed on $\bbZ^d$. This could also be seen as a system of non-interacting ideal polymer chains.\\

{\it \noindent Comments on the results.} First of all notice that we investigated a phase transition in $t$.
It would also be possible to play with the intensity $\lambda$
instead.
Indeed, multiplying the intensity $\lambda$ by a factor $\eta$ changes the typical
distance between two Poisson points by a factor $\eta^{-1/d}$. Thus, by scale invariance
of Brownian motion, the percolative behaviour of the model
is the same when we consider the Brownian paths up to time $\eta^{-2/d}t$ instead.
Hence, tuning $\lambda$ boils down to tuning $t$.

\par Moreover, it is worthwhile mentioning that Theorem \ref{thm:Aexist}
is stated only in the case $r=0$, which is the case of interest to us.
The result is the same when $r>0$, up to minor modifications.  
However, if $d\geq 4$ the paths of two independent $d$-dimensional 
Brownian motions starting at different points do not intersect.
Hence, in this case $r$ has to be chosen positive, otherwise
no percolation phase transition occurs.

\par We finish with a complementary result to Theorem \ref{thm:Bexist}: if $d\geq4$ and $r$ is such that $\gl > \gl_c(\gd_r)$, then $\cO_{0,r}$ already contains an unbounded component; therefore there is percolation at all times. In that case, van den Berg, Meester and White~\cite{vdBMW97} proved a stronger result: almost-surely, for all $t\geq0$, the set $\cup_{x\in\cE} \cB(B_t^x,r)$ contains an unbounded component.\\

{\it \noindent Open questions.} The results proven in this article answer the
first questions typically asked when studying a new percolation model.
However, there are still many challenges left open. We mention some of them:\\
{\bf \noindent (1)} How does the vacant set, that is the complement of $\cO_{t,r}$ in $\bbR^d$, look like? For instance, what is the tail behaviour of the distance from the origin to $\cO_{t,r}$?\\
{\bf \noindent (2)} What is the behaviour of $t_c(r)$ as $r\downarrow0$ for $d\geq4$?\\
{\bf \noindent (3)} How rigorous can one make the relation to random interlacement?\\ 
{\bf \noindent (4)} How rigorous can one make the relation to the system of independent finite-time random walks, which are initially homogeneously distributed on $\Z^d$?\\
{\bf \noindent (5)} If $d\geq 4$, what happens if the radii of the Wiener sausages decrease with time?\\
{\bf \noindent (6)} Is there percolation at criticality?\\
Question (6) is probably the most challenging. Question (2) is tackled in \cite{EP15}.\\

{\it \noindent Sketch of the proofs.}
$\bullet$  The main idea to prove non-percolation at small times is to dominate $\cO_{t,r}$ by a Boolean  percolation model with radius distribution given by the maximal displacement of a Brownian motion before time $t$. Standard results on the Boolean model yield non-percolation at small times.

It is important to mention that in the case $d\geq4$, additional work is required. Indeed, we need to discard the possibility that (i) $\lambda$ is supercritical for all $t>0$ and (ii) $\gl$ is subcritical at $t=0$, which means proving continuity of the critical intensity of the Boolean model w.r.t.\ the radius distribution at $\gd_r$.
This is obtained in Proposition~\ref{lem:continuity}, which requires a renormalization procedure (see Lemma \ref{lem:renormalization}) and
extends a finite box criterion for non-percolation in the Boolean model to radius distributions with an exponential tail.
To our knowledge such a criterion has only been proved for bounded radii. Moreover, we suspect that this could be extended to radius distributions with sufficiently thin polynomial tails.

$\bullet$ To establish the existence of a percolation phase, we distinguish between two cases:\\
{\bf \noindent (1)} For $d\in\{2,3\}$, we use a coarse-graining argument. 
More precisely, we divide $\R^d$ into boxes and we consider an edge percolation model of the coarse-grained graph whose vertices are identified with the centers of the boxes and the edges connect nearest neighbours. 
An edge connecting nearest neighbours, say $x$ and $x'$ in $\bbZ^d$, is said to be open if (i) both boxes associated to $x$ and $x'$ contain at least one point of the Poisson point process, say $y$ and $y'$, and (ii) the Brownian motions starting from $y$ and $y'$ intersect each other. 
A domination result by Liggett, Schonmann and Stacey \cite{LSS97} finally shows that percolation in that coarse-grained model occurs if one suitably chooses the size of the boxes and let time run for long enough. 
This implies percolation of our original model.\\
{\bf \noindent (2)} For $d\geq 4$, our strategy is to construct a $(d-1)$-dimensional supercritical Boolean model included in $\cO_{t,r}$.

$\bullet$ The difficulty in the uniqueness proof lies in extending the Burton-Keane argument to the continuous setting. For this purpose, we exploit ideas from Meester and Roy~\cite{MR94, MR96}. The case $d=3$ turns out to be the most delicate one and requires new ideas such as a careful cutting-and-glueing procedure on the Brownian paths.

\subsection{Outline of the paper}
\label{S1.5}
\noindent
We shortly describe the organization of the article.
In Section \ref{S2} we introduce the Boolean percolation model and prove some of its properties.
In Section \ref{S3} we prove Theorem \ref{thm:noperc}.
The proofs of Theorems \ref{thm:Aexist} and \ref{thm:Bexist} are 
given in Sections \ref{S4}--\ref{S6}.
Section \ref{S4} (resp. \ref{S5}) deals with the existence of a non-percolation (resp. percolation) phase.
In Section \ref{S6}
the uniqueness of the unbounded cluster is established.
The appendix provides a proof of a technical lemma which is needed in Section
\ref{S2}.


\section{Preliminaries on Boolean percolation}
\label{S2}
The model of Boolean percolation has been discussed in great detail in
Meester and Roy \cite{MR96} and we refer to this source for a 
discussion which goes beyond the description we are giving here.

\subsection{Introduction to the model}
\label{S2.1}
Let $\rho$ be a probability measure on $[0,\infty)$ and let $\chi$ be
a Poisson point process on $\R^d\times [0,\infty)$ with intensity
$(\lambda\times \Leb_d)\otimes\rho$. We denote the corresponding
probability measure by $\P_{\lambda, \rho}$.
A point $(x,r(x))\in\chi$ is interpreted to be the open ball in $\R^d$ with center $x$ and radius $r(x)$.
Furthermore, we let $\cE$ be the projection of $\chi$ onto $\R^d$. For $A\subseteq \bbR^d$, let
\begin{equation}
\Sigma(A) = \bigcup_{x\in\cE \cap A} \cB(x, r(x)).
\end{equation}
Boolean percolation deals with properties of the random set $\Sigma := \Sigma(\bbR^d)$.
We denote by $C(y)$, with $y\in\R^d$, the cluster of $\Sigma$
which contains $y$. If $y\notin\Sigma$, then $C(y)=\emptyset$.

\begin{theorem}[Gou\'er\'e, \cite{G08}, Theorem 2.1]
\label{thm:Gouere}
Let $d\geq 2$.
For all probability measures $\rho$ on $(0,\infty)$ the following assertions
are equivalent:\\
{\rm (a)} 
\begin{equation}
\label{eq:finitemoments}
\int_{0}^{\infty}x^d\,\rho(dx)<\infty.
\end{equation}\\
{\rm(b)} There exists $\lambda_0 \in (0,\infty)$ such that for all $\lambda <\lambda_0$,
\begin{equation}
\label{eq:noperc}
\P_{\lambda,\rho}\big(C(0)\mbox{ is unbounded}\big)=0.
\end{equation}
Moreover,
if {\rm(a)} holds, then, for some $c=c(d)>0$, (\ref{eq:noperc}) is satisfied for all
\begin{equation}
\label{eq:lambdabound}
\lambda < c\ \bigg(\int_{0}^{\infty}x^d\rho(dx)\bigg)^{-1}.
\end{equation}
\end{theorem}
\noindent
It is immediate from Theorem \ref{thm:Gouere}, that
\begin{equation}
\label{eq:lambdacrit}
\lambda_c(\rho):= \inf\big\{\lambda>0: 
\P_{\lambda,\rho}\big(C(0)\mbox{ is unbounded} \big)
>0\big\}>0.
\end{equation}
Moreover, from the remark on page 52 of \cite{MR96} it also follows
that $\lambda_c(\rho)<\infty$ if $\rho((0,\infty))>0$.
A more geometric fashion to characterize (\ref{eq:lambdacrit}) is via crossing 
probabilities.
For that fix $N_1, N_2, \ldots, N_d >0$ and for $A\subseteq \bbR^d$ let $\mathrm{CROSS}(N_1,N_2,\ldots, N_d ; A)$
be the event that the set $\Sigma(A) \cap [0,N_1]\times [0,N_2]\times\cdots \times[0,N_d]$
contains a component $\cC$ such that $\cC\cap \{0\}\times[0,N_2]\times\cdots \times[0,N_d]\neq \emptyset$ and $\cC\cap \{N_1\}\times[0,N_2]\times\cdots \times[0,N_d]\neq \emptyset$.
The critical value $\lambda_{\mathrm{CROSS}}$ with respect to this event is defined by
\begin{equation}
\label{eq:lambdaS}
\lambda_{\mathrm{CROSS}}(\rho)= \inf\left\{\lambda >0:\, 
\limsup_{N\to\infty}\P_{\lambda,\rho}\left(\mathrm{CROSS}(N,3N,\ldots,3N; \bbR^d)\right) >0\right\}.
\end{equation}
Under the assumption that $\rho$ has compact support, Menshikov, Molchanov and Sidorenko \cite{MMS88} proved that
\begin{equation}
\label{eq:equalitylambda}
\lambda_c(\rho) = \lambda_{\mathrm{CROSS}}(\rho).
\end{equation}

\subsection{Continuity of $\lambda_c(\rho)$}
\label{S2.2}
Given two probability measures $\nu$ and $\mu$ on $\R$ we write $\nu \preceq \mu$, if $\mu$ stochastically dominates $\nu$.
\begin{proposition}
\label{lem:continuity}
Let $\rho$ be a probability measure on $[0,\infty)$ with bounded support and
let $(\rho_n)_{n\in\N}$ be a sequence of probability measures on $[0,\infty)$
such that $\rho_n \to \rho$ weakly as $n\to\infty$ and $\rho \preceq \rho_n$ for each $n\in\N$.
Moreover, assume that 
\begin{itemize}
\item there are $c>0$ and $R_0>0$ such that for all $n\in\N$, 
$\rho_n([R,\infty))\leq e^{-cR}$ for all $R\geq R_0$;
\item there is a probability measure $\rho'$ on $[0,\infty)$
with a finite moment of order $d$ such that $\rho_n \preceq \rho'$ for all $n\in\N$.
\end{itemize}
Then, 
\begin{equation}
\label{eq:continuity}
\lim_{n\to\infty}\lambda_c(\rho_n) = \lambda_c(\rho).
\end{equation}
\end{proposition}
The proof of Proposition \ref{lem:continuity} relies on the following two lemmas whose
proofs are given in the appendix and at the end of this section, respectively.
\begin{lemma}
\label{lem:renormalization}
Let $N\in\N$, $\lambda>0$ and let $\rho$ be a probability measure on $[0,\infty)$ such that there are constants $c=c(\rho)>0$ and $R_0>0$ such that 
$\rho([R,\infty))\leq e^{-c R}$ for all $R\geq R_0$.
There is an $\varepsilon = \varepsilon(c,d) >0$ such that if
\begin{equation}
\label{eq:crossing}
\P_{\lambda,\rho}(\mathrm{CROSS}(N,3N,\ldots, 3N;\bbR^d))\leq \varepsilon,
\end{equation}
then $\P_{\lambda,\rho}(\exists\, y\in\R^d\, :\, \Leb_d(C(y))=\infty) = 0$.
\end{lemma}
\begin{lemma}
\label{lem:outsideBM}
Choose $\eta>0$ and $\rho'$ according to Proposition \ref{lem:continuity}, then
for all $N\in\N$
\begin{equation}
\label{eq:outsideBM}
\begin{aligned}
\lim_{M\to\infty}
\P_{\lambda,\rho'} 
&\bigg(\exists\ y\in \cpl{\cB_{\infty}(0,M)}\cap\cE \text{ s.t.\ } \cB(y,r(y))\cap[0,N]\times[0,3N]^{d-1}\neq \emptyset\bigg) =0.
\end{aligned}
\end{equation}
\end{lemma}
We start with the proof of Proposition \ref{lem:continuity} subject to Lemmas
\ref{lem:renormalization}--\ref{lem:outsideBM}.
\begin{proof}[Proof of Proposition \ref{lem:continuity}]
The idea of the proof is due to Penrose \cite{P95}.
First, note that 
\begin{equation}
\label{eq:limsupcrit}
\limsup_{n\to \infty}\lambda_c(\rho_n) \leq \lambda_c(\rho),
\end{equation}
since $\rho\preceq\rho_n$ for all $n\in\N$.
Thus, we may focus on the reversed direction in (\ref{eq:limsupcrit}).
Second, fix $\lambda < \lambda_c(\rho)$ and let $\varepsilon >0$ be chosen
according to Lemma \ref{lem:renormalization}. By
(\ref{eq:equalitylambda}) there is $N\in\N$ such that
\begin{equation}
\label{eq:sigmabound}
\P_{\lambda,\rho}\left(\mathrm{CROSS}(N,3N,\ldots,3N; \bbR^d)\right) \leq \varepsilon/3.
\end{equation}

We consider the following coupling $(\hat{\Omega},\hat{\P})$ of $\{\P_{\lambda, \rho_n}\}_{n \in \N}$ and $\P_{\lambda, \rho}$:
\begin{itemize}
\item the points of $\cE$ are sampled according to $\P_\lambda$;
\item by Skorokhod's embedding theorem, for each $x\in \cE$, the radii $\{r_n(x)\}_{n \in \N}$ and $r(x)$
can be coupled in such a way that they have respective distributions $\{\rho_n\}_{n\in\N}$ 
and $\rho$, and $r_n(x) \xrightarrow[n \to
\infty]{} r(x)$ a.s.
\end{itemize}
The configurations obtained via this coupling are denoted by
\begin{equation}
\quad \Sigma_n := \bigcup_{x\in\cE}\cB(x,r_n(x)),\ n\in\N, \quad \mbox{ and }\quad  \Sigma_\infty := \bigcup_{x\in\cE}\cB(x,r(x)).
\end{equation}
Let $M>0$ and consider the events
$$E_n=\{\hat{\Sigma}:=(\Sigma_k)_{k \in \N \cup \{\infty\}} \ : \ \Sigma_n \in \mathrm{CROSS}^M\}, \quad n \in \N \cup \{\infty\},$$
where 
\begin{equation*}
\mathrm{CROSS}^M= \mathrm{CROSS}(N,3N,\ldots, 3N; \cB_\infty(0,M)).
\end{equation*}
Since the number of points in $\cB_{\infty}(0,M)\cap\cE$ is finite a.s., we may conclude that
\begin{equation}
\label{eq:convXrt}
\lim_{n\to \infty} \one_{E_n} = \one_{E_{\infty}} \qquad \text{a.s.}
\end{equation}
Note that the convergence in (\ref{eq:convXrt}) is not true for every possible realization, but indeed on a set of probability one.
Hence, by the dominated convergence theorem,
$$\lim\limits_{n \to \infty}\hat{\P}(E_n)=\hat{\P}(E_{\infty}).$$
Therefore,
$$\lim_{n \to \infty}\P_{\lambda,\rho_n}(\mathrm{CROSS}^M)= \P_{\lambda, \rho}(\mathrm{CROSS}^M),$$
so that for all $n\in\N$ large enough,
\begin{equation}
\label{eq:Xrtproba}
\P_{\lambda,\rho_n}(\mathrm{CROSS}^M)\leq 2\varepsilon/3.
\end{equation}
Whence, Lemma \ref{lem:outsideBM} and the fact that $\rho_n\preceq \rho'$ for all $n\in\N$, yields
that there is $n_0\in\N$ such that for all $n\geq n_0$,
\begin{equation}
\label{eq:sigmafinal}
\P_{\lambda,\rho_n}\big(\mathrm{CROSS}(N,3N,\ldots,3N;\bbR^d) \big)\leq \varepsilon.
\end{equation}
Thus, as a consequence of Lemma \ref{lem:renormalization}, there is no unbounded component under $\P_{\lambda, \rho_n}$ for all $n\geq n_0$.
Consequently, $\lambda < \lambda_c(\rho_{n})$ for all $n\geq n_0$,
from which Proposition \ref{lem:continuity} follows.
\end{proof}

The proof of Lemma \ref{lem:renormalization} is given in Appendix \ref{A1}.

\begin{proof}[Proof of Lemma \ref{lem:outsideBM}]
Fix $M>0$ and divide $\cpl{\cB_{\infty}(0,M)}$ into a disjoint family of annuli. Basic properties of Poisson point processes and a straightforward calculations yield the result. We omit the details. 
\end{proof}

\section{Proof of a non-percolation phase}
\label{Section3}

In this section we denote by $\rho_{t,r}$ the law of the random variable $\sup_{0\leq s\leq t} \|B_{s}^0\| + r$, and $\rho_t = \rho_{t,0}$. Let us also define
\begin{equation}
\label{eq:sigmat}
\Sigma_{t,r} = \bigcup_{x \in \cE} \cB\Bigg(x,4\sup_{0\leq s\leq t} \|B_{s}^{x}-x \| + r\Bigg)
\end{equation}
and observe that
\begin{equation}
\label{eq:inclusion}
\cO_{t,r} \subseteq \Sigma_{t,r}.
\end{equation}

\subsection{Proof of Theorem \ref{thm:noperc}}
\label{S3}
Let $t>0$. Note that $\Sigma_t$ has the same law as the occupied set in the Boolean percolation model with radius distribution
$\rho_{2t}$. Basic properties of Brownian motion show that $\rho_{2t}$ has a finite moment of order $d$.
Thus, by Theorem 3.1 in \cite{MR96}, almost-surely, the set $\Sigma_t$ does not contain an unbounded cluster.
Finally, the inclusion in \eqref{eq:inclusion} yields the result.


\subsection{Theorems \ref{thm:Aexist}-\ref{thm:Bexist}:  no percolation for small times}
\label{S4}
In this section we show that there is a $t_c=t_c(\lambda,d)>0$ ($t_c=t_c(\lambda, d,r)>0$ when $d\geq 4$) such that
$\cO_t$ ($\cO_{t,r}$ when $d\geq 4$) does not percolate when $t<t_c$.
The proof for $d\in\{2,3\}$ appears in Section \ref{S3.1}, whereas the proof for $d\geq 4$ appears in Section \ref{S3.2}.
Both proofs rely on the results of Section \ref{S2}.

\subsubsection{No percolation for $d\in\{2,3\}$}
\label{S3.1}
Recall (\ref{eq:lambdabound}) in Theorem \ref{thm:Gouere}. The inclusion in \eqref{eq:inclusion} and the fact that 
\begin{equation}
\label{eq:zero2ndmom}
\lim_{t\to 0} \int_{0}^{\infty}x^d\,\rho_{2t}(dx) =0
\end{equation}
are enough to conclude.


\subsubsection{No percolation for $d\geq 4$}
\label{S3.2}
Note that $\rho_{2t,r} \to \delta_r$ weakly as $t \to 0$. Moreover, one readily checks that the assumption of Proposition \ref{lem:continuity} are met (with $\rho'= \rho_{1,r}$), therefore $\lambda_c(\rho_{2t,r}) \to \lambda_c(\delta_r)$ as $t\to 0$. Hence, there is a $t_0 >0$ such that
$\lambda < \lambda_c(\rho_{2t,r})$ holds for all $t< t_0$.
Finally, we conclude with \eqref{eq:inclusion}.

\section{Theorems \ref{thm:Aexist}--\ref{thm:Bexist}:  percolation for large times}

\label{S5}
In this section we establish that $\cO_t$ ($\cO_{t,r}$ when $d\geq 4$) percolates,
when $t$ is sufficiently large. The proof for $d\in\{2,3\}$ appears in Section \ref{S4.1},
whereas the proof for $d\geq 4$ appears in Section \ref{S4.2}.

\subsection{Proof of the percolation phase in $d\in\{2,3\}$}
\label{S4.1}

The proof proceeeds according to the strategy described at the end of Section \ref{S1.4}, which relies on the introduction of a coarse-grained model. We now define this coarse-grained model more rigorously. Let $R>0$ and $t>0$ to be chosen later. Fix $x\in\bbZ^d$.
When $\mid \cE \cap \cB_{\infty}(2Rx,R) \mid \geq 1$, we define the point $z^{(R,x)}$, which is almost surely uniquely determined, via
\begin{equation}
\| z^{(R,x)} - 2Rx \|  = \inf_{z\in \cE \cap \cB_{\infty}(2Rx,R)}\| z - 2Rx \|.
\end{equation}
We denote by $B^{(R,x)}$ the Brownian motion starting at $z^{(R,x)}$. For all pairs of
nearest neighbours $(x,y)\in \bbZ^d \times \bbZ^d$, we say that the edge $\{x,y\}$,
which connects $x$ and $y$, is open if 
\begin{align}
\rm{(i)}\quad &\mid \cE \cap \cB_{\infty}(2Rx,R) \mid \geq 1,\\
\rm{(ii)}\quad &\mid \cE \cap \cB_{\infty}(2Ry,R) \mid \geq 1 \quad \text{and}\\
\rm{(iii)}\quad &\,B^{(R,x)}_{[0,t]} \cap B^{(R,y)}_{[0,t]} \neq \emptyset.
\end{align}  
We let $X_{\{x,y\}} = \one\{\text{the edge } \{x,y\} \text{ is open}\}$. We omit the dependence on $R$ and $t$ not to burden the notation.\\

\begin{lemma}\label{lem:perc_cgmodel}
Let $\epsilon>0$. There exists $R>0$ and $t>0$ such that for any couple of nearest neigbours $(x,y)\in \bbZ^d \times \bbZ^d$, $\P(X_{\{x,y\}}=1)\geq 1-\epsilon$.
\end{lemma}
The proof of Lemma \ref{lem:perc_cgmodel} is deferred to the end of this section. We first show how one deduces the existence of a percolation phase from it.
\begin{proof}[Proof of the existence of a percolation phase]
Note that if $(x,x')$ and $(y,y')$ is a pair of nearest neighbour points in $\bbZ^d$ such that $\{x,x'\}\cap \{y,y'\} = \emptyset$, then $X_{\{x,x'\}}$
and $X_{\{y,y'\}}$ are independent. Therefore, the coarse-grained percolation model is a $2$-dependent percolation model. Thus, Theorem 0.0 of Liggett, Schonmann
and Stacey \cite{LSS97} yields that we may stochastically minorate the coarse-grained percolation model by a Bernoulli bond percolation model, whose parameter, say $p^*$, can be
chosen arbitrarily close to $1$, provided $\P(X_{\{x,y\}}=1)$ is sufficiently close to $1$. Let $ p_c(\bbZ^d)$ be the critical percolation parameter for Bernoulli bond percolation. Then, by Lemma \ref{lem:perc_cgmodel}, there are $R_0>0$ and
$t_0>0$ such that $p^* > p_c(\bbZ^d)$ for all $R\geq R_0$ and $t\geq t_0$. In that case, the coarse-grained model percolates, and so does $\cO_t$.
\end{proof}
\noindent
Consequently, it remains to prove Lemma \ref{lem:perc_cgmodel}.

\begin{proof}[Proof of Lemma \ref{lem:perc_cgmodel}]
By independence of the events in (i)--(iii), we have
\begin{equation}
\label{eq:ingr1}
\P(X_{\{x,y\}}=1) = \E \Big[  \one \Big\{
\begin{array}{c}
 \mid \cE \cap \cB_{\infty}(2Rx,R) \mid \geq 1  \\
\mid \cE \cap \cB_{\infty}(2Ry,R) \mid \geq 1
\end{array}
\Big\}
 \P\Big(B^{(R,x)}_{[0,t]} \cap B^{(R,y)}_{[0,t]}
\neq \emptyset \big|\ \cE\ \Big) \Big] .
\end{equation}
To proceed, we fix $R>0$ large enough such that 
\begin{equation}
\label{eq:ingr2}
\P( \mid \cE \cap \cB_{\infty}(2Rx,R) \mid \geq 1) = 1 - e^{-\lambda (2R)^d} \geq 1-\epsilon.
\end{equation}
Furthermore, $\P(B^{(R,x)}_{[0,t]} \cap B^{(R,y)}_{[0,t]} \neq\emptyset |\ \cE\ )$ decreases when $\|z^{(R,x)} - z^{(R,y)}\|$
increases and
$\|z^{(R,x)}-z^{(R,y)}\| \leq  R\sqrt{4(d-1)+16}$ when $\|x-y\|=1$. Thus, 
\begin{align}
\P\Big(B^{(R,x)}_{[0,t]} \cap B^{(R,y)}_{[0,t]} \neq \emptyset \Big|\ \cE\ \Big) &\geq \P\Big(B^{(R,x)}_{[0,t]} \cap B^{(R,y)}_{[0,t]} \neq \emptyset\Big| \|z^{(R,x)} - z^{(R,y)}\|
= R\sqrt{4(d-1)+16}\Big)\\
& = \bbP^{z_1,z_2}\Big(B^{(1)}_{[0,t]} \cap B^{(2)}_{[0,t]} \neq \emptyset\Big),
\end{align}
for any choice of $z_1$ and $z_2$ such that $\|z_1 - z_2\| = R\sqrt{4(d-1)+16}$. By Theorem 9.1 (b) in M\"orters and Peres \cite{MP10}, there exists $t$
large enough such that for all such choices of $z_1$ and $z_2$,
\begin{equation}
\label{eq:ingr3}
\bbP^{z_1,z_2}\Big(B^{(1)}_{[0,t]} \cap B^{(2)}_{[0,t]} \neq \emptyset\Big) \geq 1 - \epsilon.
\end{equation} 
The combination of \eqref{eq:ingr3}, \eqref{eq:ingr3} and \eqref{eq:ingr3} yields the result.
\end{proof}

\subsection{Proof of the percolation phase for $d\geq 4$}
\label{S4.2}
Throughout the proof, $z$ always denotes the $d$-th coordinate of $x=(\xi,z)\in
\R^d$. We further define
\begin{equation}
\pazocal{H}_0=\{(\xi,z)\in\R^d:\, z=0\}.
\end{equation}
\noindent
The main idea is to show percolation for a Boolean model on $\pazocal{H}_0$.
More precisely, we use that for each $x\in\pazocal{E}$, $B^{x}$ will eventually hit $\pazocal{H}_0$. From this we deduce that for $t$ large enough, the traces
of the Wiener sausages which hit $\pazocal{H}_0$ dominate a supercritical $(d-1)$-dimensional Boolean percolation model, and therefore percolate.\\

\noindent
We now formalize this strategy. In this proof, we write a $d$-dimensional Brownian motion $B$ as $(B^\text{I},B^\text{II})$ where $B^\text{I}$ and $B^\text{II}$ stand for a one and $(d-1)$-dimensional standard Brownian motion respectively.
For each $k\in\N$, let
\begin{equation}
\pazocal{S}_k:=\{(\xi,z) \in \R^d \ : \ k-1 < z \leq k \},
\end{equation}
so that $(\cS_k)_{k\in \bbZ}$ is a partition of $\bbR^{d-1}\times (0,\infty)$.
We fix $k\in \N$ and consider
\begin{equation}
\pazocal{E}_k=\{ \xi: \exists \ z\in \R \ \mbox{s.t.} \ (\xi,z) \in \pazocal{S}_k\cap \pazocal{E} \}.
\end{equation}
Note that $(\pazocal{E}_k)_{k\geq 0}$ are i.i.d. Poisson point processes with parameter $\lambda\times \mathrm{Leb}_{d-1}$. Given $\pazocal{E}_k$, we construct
a random set $\pazocal{P}^k_t$ in the following way:
\begin{itemize}
\item Thinning: each $\xi \in \pazocal{E}_k$ is kept if $\tau_0(z^\xi)\leq t$, where $z^\xi$ is such that $(\xi, z^{\xi})\in \cS_k \cap \cE$ (there is
almost-surely only one choice), and $\tau_0(z)$ is the first hitting time of the origin by a one-dimensional Brownian motion starting at
$z$. We choose all Brownian motions to be independent. 
Otherwise, $\xi$ is discarded. 

\item Translation: each $\xi \in \pazocal{E}_k$ that was not discarded after the previous step is translated by $B^\text{II}(\tau_0(z^\xi))$.
\end{itemize}
Note that $z^\xi$ is uniformly distributed in $(k-1,k)$. Moreover, $z^{\xi}$, $\tau_0(z^\xi)$ and $B^\text{II}$ are independent of $\xi$.
Thus, $\pazocal{P}^k_t$ is the result of a thinning and a translation of $\cE_k$ and both operations depend on random variables which are independent of
$\pazocal{E}_k$. Therefore, $(\pazocal{P}^k_t)_{k \geq0}$ is a collection of i.i.d. Poisson point processes with parameter $\lambda p^k_t\times \mathrm{Leb}_{d-1}$, where
\begin{equation}
p^k_t=\int_{k-1}^{k}\bbP^{0}\Big(\inf_{0\leq s\leq t}B^\text{I}_s\leq -z\Big)\, dz
\geq 
\bbP^0\Big(\sup_{0 \leq s \leq t} B^\text{I}_s\geq k\Big).
\end{equation}
By independence of the $\pazocal{P}_t^k$'s, the set $\pazocal{P}_t:=\bigcup_{k=1}^\infty
\pazocal{P}^k_t$ is a Poisson point process with parameter $\lambda\sum_{k\geq1} p^k_t\times \mathrm{Leb}_{d-1}$.

Let us now consider the Boolean model generated by $\pazocal{P}_t$ with deterministic radius $r$.
Observe that, 
\begin{equation}
\label{ls}
\sum_{k=1}^{\infty} p^k_t \geq \sum_{k=0}^\infty \bbP^0\Big(\sup_{0 \leq s \leq t} B^\text{I}_s\geq k\Big) - \bbP^0\Big(\sup_{0 \leq s \leq t} B^\text{I}_s\geq 0\Big) \geq 
\bbE^0\Big[\sup_{0 \leq s \leq t} B^\text{I}_s\Big]-1.
\end{equation}

Note that the right-hand side of \eqref{ls} tends to infinity as $t\to \infty$.
Thus, by the remark on page $52$ in \cite{MR96}, there exists $t_0>0$ large enough such that the Boolean model generated by $\pazocal{P}_t$ percolates for all $t\geq t_0$. 
Finally, note that $\pazocal{P}_t$ is stochastically dominated by $\cO_t \cap \cH_0$, in the sense that $\pazocal{P}_t$ has the same distribution as a subset of $\cO_t \cap \cH_0$. 
This completes the proof.

\section{Theorems \ref{thm:Aexist}--\ref{thm:Bexist}:  uniqueness of the unbounded cluster}
\label{S6}

We fix $t,r,\lambda\geq 0$ such that $t>t_c(\lambda,d,r)$.
In the following we denote by $N_\infty$ the number of unbounded clusters in $\cO_{t,r}$, which is almost-surely a constant as a
consequence of Remark~\ref{rem:ergodic}. For all $d\geq2$, the proof of uniqueness consists of (i) excluding the case $N_\infty = k$ with $k \in\bbN\setminus\{1\}$ and (ii) excluding the case $N_\infty= \infty$.
Section \ref{S5.2} contains the proof of uniqueness for
Wiener sausages ($r>0$) in $d\geq 4$, whereas Section~\ref{S5.3} contains the proof of uniqueness in $d\in\{2,3\}$.

\subsection{Uniqueness in $d\geq 4$}
\label{S5.2}

\subsubsection{Excluding $2\leq N_{\infty} < \infty$}
\label{S5.2.1}

In what follows we write for each $A\subseteq \mathbb{R}^d$,
\begin{equation}
\cO_{t,r}(A) = \bigcup_{x\in\cE\cap  A} \bigcup_{0\leq s \leq t} \cB(B_s^x, r),
\end{equation}
which is the union of Wiener sausages started at points of $\cE$ restricted to $A$.\\

We proceed by contradiction. Let us assume that $N_{\infty}$ is almost-surely equal to a constant $k\in\bbN\setminus\{1\}$.\\

\noindent
For $R_2 > R_1 > 0$, let us define $E_{R_1, R_2}$ as follows:
\begin{equation}
\label{al:ER1R2bis}
E_{R_1, R_2} =  \left\{
\mbox{all unbounded clusters of $\cO_{t,r}(\cpl{\cB(0,R_1)})$} 
\mbox{ intersect } \cB(0,R_2)\right\}.
\end{equation}
\noindent First, we note that there exist $R_1$ and $R_2$ such that
\begin{equation}
 \P(E_{R_1, R_2}) > 0.
\end{equation}
Indeed, fix $R_1>0$ and note that by monotonicity in $R_2$,
\begin{equation}
\label{ER1R2pos}
\P(E_{R_1, R_2}) \geq \P(E_{R_1, R_2} \cap \{ \cE \cap \cB(0,R_1) = \emptyset \}) \stackrel{R_2 \to \infty}{\longrightarrow} \P(\cE \cap \cB(0,R_1) = \emptyset)
>0.
\end{equation}
Therefore, we can find $R_2> 0$ such that $\P(E_{R_1, R_2})>0$.
\noindent
Next, we consider the event,
\begin{equation}
\label{ER1mod}
L_{R_1,R_2}=
\left\{
\begin{array}{c}
|\cB(0,R_1)\cap \cE|=1 \mbox{ and for }  x\in \cB(0,R_1)\cap \cE,\\ 
\cA(R_2-3r/2, R_2-r/2) \subset \bigcup_{0\leq s \leq t} \cB(B_s^x,r)\subset \cB(0,R_2)
\end{array}
\right\},
\end{equation}
which is independent of $E_{R_1,R_2}$ and has positive probability, see Remark \ref{rem:LR1R2} below. 
The independence is due to the fact that $E_{R_1,R_2}$ and $L_{R_1,R_2}$ depend on different points of $\cE$ and on different Brownian paths. 
Note that on $E_{R_1,R_2}\cap L_{R_1,R_2}$ all unbounded clusters of $\pazocal{O}_{t,r}(\cpl{\cB(0,R_1)})$ are connected inside $\cB(0,R_2)$. This is enough to conclude the proof.

\begin{remark} 
\label{rem:LR1R2}
A sketch of the proof that $L_{R_1,R_2}$ has positive probability goes as follows. Let $\epsilon\in (0,r/8)$. By boundedness, $\cA(R-3r/2,R_2-3r/2+\epsilon)$ can be covered by
a finite number of balls of radius $\epsilon$. Moreover, a Brownian motion starting in $\cB(0,R_1)$ has a positive probability of visiting all these balls before
time $t$ and before leaving $\cB(0,R_2-r)$. Consequently, on the aforementioned event,  $L_{R_1,R_2}$ is satisfied.
\end{remark}

\subsubsection{Excluding $N_{\infty} = \infty$}
\label{S5.2.2}

We assume that $N_{\infty}=\infty.$
We show that this assumption leads to a contradiction.
The proof is based on ideas in Meester and Roy~\cite[Theorem 2.1]{MR94}, where
a technique developed in Burton and Keane~\cite{BK89} is extended 
to a continuous percolation model. In the proof we use the following counting lemma, which
is due to Gandolfi, Keane and Newman~\cite{GKN92}. 

\begin{lemma}[Lemma 4.2 in \cite{GKN92}]
\label{lem:counting}
Let $\pazocal{S}$ be a set, $\pazocal{R}$ be a non-empty finite subset of $\pazocal{S}$ and $K>0$.
Suppose that\\
{\rm (a)} for all $z\in \pazocal{R}$, there is a family $(C_z^{1},C_z^{2},\ldots, C_{z}^{n_z})$, $n_z\geq 3$,
of disjoint non-empty subsets of $\pazocal{S}$, which do not contain $z$ and are such that
$|C_z^{i}|\geq K$, for all $z$ and for all $i\in\{1,2,\ldots,n_z\}$,\\
{\rm (b)} for all $z,z' \in \pazocal{R}$ one of the following cases occurs
(where  we abbreviate $C_z= \cup_{i=1}^{n_z}C_z^{i}$ for all $z\in \pazocal{R}$):\\
{\rm (i)} $(\{z\}\cup C_z)\cap (\{z'\}\cup C_{z'}) = \emptyset$;\\
{\rm (ii)} there are $i,j \in\{1,2,\ldots,n_z\}$ such that
$\{z'\}\cup  C_{z'}\setminus C_{z'}^{j} \subseteq C_z^{i}$ and
$\{z\}\cup  C_{z}\setminus C_{z}^{i} \subseteq C_{z'}^{j}$;\\
{\rm(iii)} there is $i\in\{1,2,\ldots,n_z\}$ such that $\{z'\}\cup C_{z'}\subseteq
C_z^{i}$;\\
{\rm(iv)} there is $j\in\{1,2,\ldots,n_{z'}\}$ such that $\{z\}\cup C_z\subseteq C_{z'}^{j}$.\\
\medskip\noindent
Then, $|\pazocal{S}| \geq K(|\pazocal{R}|+2)$.
\end{lemma}

\paragraph{STEP 1. Preparation for Lemma~\ref{lem:counting}.}

In the same manner as in Section \ref{S5.2.1}, one can show that there are $\delta >0$ and $R\in\N$ such that the event
\begin{equation}
\label{eq:EN0}
E_R(2Rz) := 
\left\{
\parbox{11cm}
{
there exists an unbounded cluster $C$ such that 
$C\cap \cpl{\cB_{\infty}(2Rz,R)}$ contains at least three unbounded clusters,
$|C \cap \cB_{\infty}(2Rz,R) \cap \cE|\geq 1$ and any cluster which intersects $\cB_{\infty}(2Rz,R)$ belongs to $C$
}
\right\}
\end{equation}
has probability at least $\delta$, for all $z\in\bbZ^d$.
\noindent
We call each unbounded cluster in $C\cap \cpl{\cB_{\infty}(2Rz,R)}$ a branch. 
To proceed, we fix $K>0$ and choose $M>0$ such that the event
\begin{equation}
\label{eq:ENM}
E_{R,M}(2Rz)= E_R(2Rz) \cap 
\left\{
\parbox{10cm}
{
there are at least three different branches of $\cB_{\infty}(2Rz,R)$ which contain at least $K$ points in $\cE\cap (\cB_{\infty}(2Rz,RM)\setminus \cB_{\infty}(2Rz,R))$
}
\right\},
\end{equation}
has probability at least $\delta/2$ for all $z\in\bbZ^d$, see Fig.~\ref{excinfinity} below. 

\begin{figure}[htbp]
\begin{center}
\includegraphics[height=6cm]{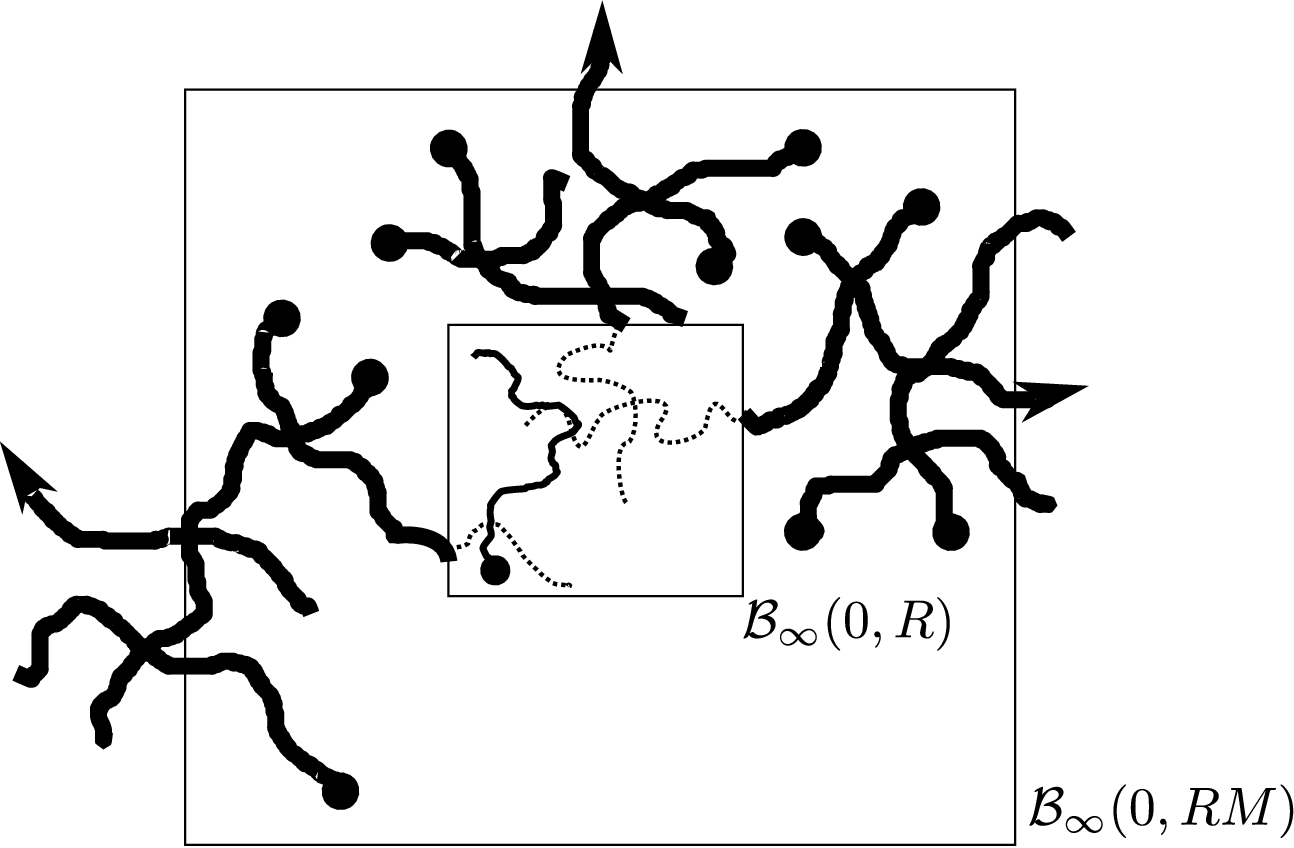} 
\end{center}
\caption{\small
The plot represents a configuration in $E_{R,M}(0)$ with $K=3$, see \eqref{eq:EN0}-\eqref{eq:ENM}.
The thick lines belong to the branches. 
The symbol $\blacktriangleright$ indicates a connection to infinity.
}
\label{excinfinity}
\end{figure}
\noindent
Let $L> M+2$ and define the set
\begin{equation}
\label{eq:R}
\cR =\{z\in\Z^d:\, 
\cB_\infty(2Rz,RM) \subseteq \cB_\infty(0,LR),\,
 E_{R,M}(2Rz) \mbox{ occurs}\}\footnotemark[1]\footnotetext[1]{The elements of $\pazocal{R}$ play the role of trifurcation points in the discrete percolation setting.}.
\end{equation}
Note that 
\begin{equation}
\label{eq:cardR}
|\{z\in\Z^d:\, 
\cB_\infty(2Rz,RM) \subseteq \cB_\infty(0,LR)\}| \geq (L-M-2)^d,
\end{equation}
so that we obtain by stationarity 
\begin{equation}
\label{eq:expcardR4}
\E(|\cR|) \geq \frac{(L-M-2)^d\delta}{2}.
\end{equation}

\paragraph{STEP 2. Application of Lemma \ref{lem:counting} and contradiction.} 

\noindent
We identify each $z\in \cR$ with a Poisson point in $\cB_\infty(2Rz,R)\cap C$.
In what follows we write $\Lambda_z$ instead of $\cB_\infty(2Rz,R)$.
Let $n_z$ be the total number of branches of $\Lambda_z$ which contain at least $K$ Poisson points in $\cB_{\infty}(2Rz,R)$. For $i\in\{1,\ldots,n_z\}$, let $\mathbf{B}^i_z$ be the branch which is the $i$th-closest to
$2Rz$ among all branches of $\cB_{\infty}(2Rz,R)$, see (\ref{eq:ENM}).

A point $x$ is said to be connected to a set $A$ \emph{through} the set $\Lambda$ if there exists a continuous function $\gamma: [0,1]\mapsto \Lambda \cap
\cO_{t,r}$ such that $\gamma(0) = x$ and $\gamma(1) \in A$. We denote it by $x\stackrel{\Lambda}{\longleftrightarrow}A$. Finally, we define  
\begin{equation}
\label{eq:Cz}
C_z^{i} = \cE \cap \cB(0,LR) \cap \mathbf{B}^i_z =
\left\{
x \in \pazocal{E}\cap\cB_{\infty}(0,LR) \colon \ \con{x}{\Lambda_z^c}{\mathbf{B}^i_z}
\right\},
\qquad 
i\in \{1, \ldots, n_z\}.
\end{equation}

\noindent
Now we proceed to check that the conditions of Lemma \ref{lem:counting} are fulfilled. Here $\pazocal{S}=\cB_\infty(0,LR)\cap \cE$.
First note that by the definition of a branch, we have that for all $z\in \cR$:
\begin{itemize}
\item $|C_z^{i}|\geq K$,
\item $C_z^{i}\cap C_z^{j} = \emptyset$ for all $i,j \in\{1,\ldots,n_z\}$ with $i\neq j$ and
\item $z \notin C_z$.
\end{itemize}
Hence, Assumption {\rm (a)} of Lemma \ref{lem:counting} is met.\\

We now claim that the collection $\{C_z^{i}\}_{z\in \cR, i\in\{1,\dots,n_z\}}$ satisfies also Assumption {\rm (b)} of Lemma \ref{lem:counting}. At this point we
would like to emphasize two facts to be used later:
\begin{description}
\item[a.] 
\label{inside}
Due to \eqref{eq:EN0}, $\con{z}{\Lambda_z}{C^i_z}$ for all $i\in\{1,\dots,n_z\}$.
\item[b.] 
\label{ic}
If $\tilde{C}$ is an unbounded cluster such that $\tilde{C}\cap\Lambda_z\neq \emptyset$, then $\con{z}{\Lambda_z}{\tilde{C}}$.
\end{description}

Suppose that $(\{z\} \cup C_z ) \cap (\{z'\} \cup C_{z'} )\neq \emptyset$.
We consider three different cases:
\begin{enumerate}
\item If $z' \in C_z$ then there exists a unique $i\in \{1,\ldots,n_z\}$ such that $z'\in C^{i}_{z}$. We consider two subcases:
\begin{itemize}
\item 
If $z \in C_{z'}$, then there exists a unique $i'\in \{1,\ldots,n_{z'}\}$ such that $z\in C^{i'}_{z'}$ and we claim that $\{z'\}\cup C_{z'} \setminus C^{i'}_{z'}\subseteq C^i_z$ and $\{z\}\cup C_{z}
\setminus C^{i}_{z}\subseteq C^{i'}_{z'}$. Indeed, pick $x' \in C_{z'} \setminus C^{i'}_{z'}$. Then there exists a unique $j'\neq i'$ such that
$\con{x'}{\Lambda_{z'}^c}{C^{j'}_{z'}}$.
Note that $\con{x'}{\Lambda_{z'}^c\cap\Lambda_{z}^c}{C^{j'}_{z'}}$, since otherwise, due to {\bf b.},
$\con{z}{\Lambda^c_{z'}}{C^{j'}_{z'}}$ (by first connecting $z$ to $x'$ in $\Lambda^c_{z'}$ and then $x'$ to $C^{j'}_{z'}$ in $\Lambda^c_{z'}$ ), which contradicts the uniqueness of $i'$. 

Finally, we have that $\con{x'}{\Lambda_z^c}{C^{j'}_{z'}}, \ \con{z'}{\Lambda_{z'}\subset \Lambda_z^c}{C^{j'}_{z'}}, \ \con{z'}{\Lambda^c_z}{C^i_z}$. 
A concatenation of all these paths gives $\con{x'}{\Lambda_z^c}{C^i_z}$, that is $x'\in C^i_z$. This proves the first inclusion that we claimed. The second
inclusion follows by symmetry.
\item 
If $z \notin C_{z'}$, then we claim that $\{z'\}\cup C_{z'} \subseteq C^i_z$.

Indeed, take $x' \in C_{z'}$, then there exists a unique $j'$ such that $\con{x'}{\Lambda_{z'}^c}{C^{j'}_{z'}}$.
As before we have that $\con{x'}{\Lambda_{z'}^c\cap\Lambda_{z}^c}{C^{j'}_{z'}}$ (this time the contradiction follows from $z \notin C_{z'}$). 
The conclusion follows in the same way as in the previous case.
\end{itemize}
\item If $z\in C_{z'}$, then one may conclude as in {\rm (1)}.
\item 
Suppose that there exist $i,i'$ such that $C^{i}_z\cap C^{i'}_{z'}\neq \emptyset$. 
Take $x' \in C^{i}_z\cap C^{i'}_{z'}$. Then, $\con{x'}{\Lambda_z^c}{C^i_z}$ and $\con{x'}{\Lambda_{z'}^c}{C^{i'}_{z'}}$.
We distinguish between two cases:
\begin{itemize}
\item 
The path $\con{x'}{\Lambda_z^c}{C^i_z}$ intersects $\Lambda_{z'}$: 
due to {\bf b.} we have that $\con{z'}{\Lambda_{z}^c}{C^i_z}$.
Hence $z' \in C_z$, which reduces to Case (1). 
\item 
Otherwise, $\con{x'}{\Lambda_z^c\cap\Lambda_{z'}^c}{C^i_z}$: 
due to {\bf a.}, we have $\con{z}{\Lambda_z \subset \Lambda^c_{z'}}{C^i_z}$. 
Finally, a concatenation of the previous two paths with $\con{x'}{\Lambda^c_{z'}}{C^{i'}_{z'}}$ yields that $z\in C_{z'}$, which reduces to Case (2).
\end{itemize}
\end{enumerate}

\noindent Hence, by Lemma \ref{lem:counting}
\begin{equation}
\E\big(|\cB_\infty(0,LR) \cap \cE|\big)
\geq K(\E(|\cR|) +2),
\end{equation}
so that, by \eqref{eq:expcardR4},
\begin{equation}
\label{eq:lemcounting4}
\E\big(|\cB_\infty(0,LR) \cap \cE|\big)
\geq K((L-M-2)^d \delta/2 +2).
\end{equation}
On the other hand, since $\cE$ is a Poisson point process 
with intensity measure $\lambda\times \Leb_d$,
\begin{equation}
\label{eq:PPest4}
\E \big(|\cB_\infty(0,LR)\cap \cE|\big) = \lambda (2LR)^d.
\end{equation}
Thus, combining (\ref{eq:lemcounting4}) and (\ref{eq:PPest4}), yields
\begin{equation}
\label{eq:contradiction}
\forall L>M+2,\quad K((L-M-2)^d\delta/2 + 2) \leq \lambda (2LR)^d.
\end{equation}
Note that $M$ depends on $K$, so in order to get a contradiction one can choose $L=2M$ and let $K$ go to infinity in \eqref{eq:contradiction}.

\subsection{Uniqueness in $d\in\{2,3\}$}
\label{S5.3}

\subsubsection{Excluding $\{2\leq N_\infty < \infty \}$}
\label{S5.3.1}

There is no straightforward way to adapt the proof of Section~\ref{S5.2} to the three-dimensional setting because of clear geometrical reasons: if an annulus is crossed by all the unbounded clusters then a three-dimensional Brownian motion travelling around it does not necessarily connect them. Let us briefly describe how we proceed in this case. Assume $2\leq N_\infty <\infty$. For $R$ large enough and $\epsilon$ small enough we show that, with positive probability, all the
unbounded clusters intersect $\cB(0,R)$ and contain a Brownian path crossing $\cA(R-\epsilon,R)$. Afterwards, we show that, still with positive probability, we can reroute the (say first) excursions inside $\cA(R-\epsilon,R)$ of
each of these Brownian paths such that they intersect each other and, as a consequence, merge all the unbounded clusters into a single one. This leads to the desired
contradiction, since
our construction provides a set of configurations of positive probability on which $N_\infty = 1$.\\ \par

\begin{remark}
\label{rem:twod}
It is possible to adapt the proof of Section~\ref{S5.2} to the two-dimensional setting. 
However, the forthcoming proof applies to the case of dimension two and three. 
So, we decided not to comment further on this adaptation and only present a unified argument for both cases.  
\end{remark}
We now assume $t>t_c$ and give the proof in full detail. 
To make it more accessible, we assume w.l.o.g.\ that $N_{\infty}=2$, see Remark~\ref{rem:otherK}.
Let $R>0$ and denote by $N_\infty^R$ the number of unbounded clusters in $\cO_t \setminus \cB(0,R)$, which we denote by $\{ C_i(R), 1\leq i \leq N_\infty^R\}$ (though it has little relevance, let us agree that clusters are indexed according to the order in which one finds them by radially exploring the occupied set from $0$). 
We also consider \emph{extended} clusters, defined by
\begin{equation}\label{eq:extcluster}
C_i^\ext(R) = \bigcup_{x\in\cE\,:\,B^x_{[0,t]} \cap C_i(R) \neq \emptyset} B^x_{[0,t]},
\end{equation}
i.e., $C_i^\ext(R)$ is the union of all Brownian paths up to time $t$ which have a non-empty intersection with $C_i(R)$.
\bigskip

\noindent
We define a notion of good extended cluster in five steps.
\paragraph{Definition a good extended cluster in five steps.}
Let $C^{\ext}=C^\ext(R)$ be an extended cluster.
We define the following events:

\smallskip
\noindent
{\bf STEP 1. Intersection with a large ball.} Set
\begin{equation}
E_R := \{C^\ext \cap \cB(0,R) \neq \emptyset\}.
\end{equation}
\smallskip
\noindent
{\bf STEP 2. Choice of a path in the extended cluster.} Consider
\begin{equation}
\cross = \{y\in \cE \cap C^\ext \,:\, \exists \ s\in[0,t],\, (\|y\| - R)(\|B_s^y\|- R)<0\},
\end{equation}
that is the set of points in $\cE \cap C^\ext$ whose associated Brownian motions cross $\partial \cB(0,R)$.
Note that $\cross\neq\emptyset$ on $E_R$.
Let $x$ be such that
\begin{equation}
\label{defx}
\|x \| =  \inf_{y\in \cross}  \|y\|.
\end{equation}
This way of picking $x$ is arbitrary.
Any other way would serve our purpose as well.

\smallskip
\noindent
{\bf STEP 3. First excursion through an annulus.} 
For a fixed $\epsilon >0$, consider the annulus $\cA_{R,\epsilon} := \cA(R-\epsilon, R)$. 
Define
\begin{equation}
I(x) := \one\{\inf\{s\geq 0: \| B_s^x \| = R\} <  \inf\{s\geq 0: \| B_s^x \| = R - \epsilon\}\}.
\end{equation}
We introduce the following entrance and exit times:
\begin{align}
\label{def:in.out1}
&\sigma^{\out} = \inf\{s \geq 0 : \| B_s^{x} \| = R - I(x)\epsilon\},\nonumber\\ 
&\sigma^{\ins} = \sup\{s \leq \sigma^{\out} : \| B_s^{x} \| = R + (I(x)-1)\epsilon\},
\end{align}
i.e., $B^{x}_{[\sigma^{\ins},\sigma^{\out}]}$ is the first excursion of $B^{x}$ through $\cA_{R,\epsilon}$, see Fig.~\ref{annulus} below.
The reason for this definition is that we do not want to exclude the possibility that $x$ is located inside $\cB(0,R)$. 
By choosing $\epsilon$ small enough we guarantee that the Brownian motion started at $x$ cross $\cA_{R,\epsilon}$, that is,
$\sigma^{\ins} \leq \sigma^{\out} \leq t$.
Further, we consider the event on which $B_{[0,\sigma^{\ins})}^{x}$ or $B_{(\sigma^{\out},t]}^{x}$ is already connected to $C^\ext$, i.e., we introduce                                                                                                                   
\begin{equation}
\label{econn}
E^{\conn}_{\epsilon} := \left\{ \left(B^{x}_{[0, \sigma^{\ins})} \cup B^{x}_{(\sigma^{\out},t]}\right) \cap C^\ext \neq \emptyset \right\}.
\end{equation}
Summing up, we set 
\begin{equation}
E_{R,\epsilon} = E_R \cap\{\sigma^\ins \leq \sigma^\out \leq t\} \cap E^{\conn}_{\epsilon}.
\end{equation}
\smallskip
\noindent
{\bf STEP 4. Restriction on the time spent to cross the annulus.} 
For $T\in(0,t)$ set
\begin{equation}
E_{R,\varepsilon,T}= E_{R,\epsilon} \cap \{\sigma^{\out} - \sigma^{\ins} \geq T\}.
\end{equation}
\smallskip
\noindent
{\bf STEP 5. Staying away from the boundary of the annulus during the excursion.}
Since  $\sigma^{\ins}$ is not a stopping time, the law of $B^{x}_{[\sigma^{\ins},\sigma^{\out}]}$ is not that of a Brownian motion. This is why we will work instead with $B^{x}_{[\sigma^{\ins}+\delta,\sigma^{\out}-\delta]}$ for a fixed $\delta\in (0,T/8)$  (the restriction to time $\sigma^{\out}-\delta$ is only for esthetic reasons).
This subpath, when conditioned on both endpoints, is a Brownian bridge conditioned to stay in $\cA_{R,\epsilon}$ and whose density with respect to a Brownian motion is explicit and tractable. 
For a fixed $\bep \in (0,\epsilon/2)$ set
\begin{align}
E_{R,\epsilon,T,\bep} &:= E_{R,\gep,T} \cap \left\{ B^{x}_{\sigma^{\ins} +\delta}, B^{x}_{\sigma^{\out} -\delta} \in
\cAJ\right\},
\end{align}
where $\cAJ:=\cA(R-\epsilon+\bep, R-\bep)\subset \pazocal{A}_{R,\epsilon}$.

\bigskip
\noindent
Having disposed of the notion of good extended cluster, let
\begin{equation*}
\tilde{E}_{R,\epsilon,T,\bep,n}=\Big\{N_\infty^R=n,\ C^\ext_i \in E_{R,\epsilon,T,\bep},\ 1\leq i \leq n \Big\}.
\end{equation*}
By monotonicity arguments and the initial assumption that $N_\infty=2$, there exist positive constants $R,T,c,\bar{\epsilon}<\epsilon/2$ and $n_0\geq 2$ such that 
\begin{equation}
\label{positive}
\P(\tilde{E}_{R,\epsilon,T,\bep,n_0})>c>0.
\end{equation}
For simplicity we consider $n_0=2$, see Remark \ref{rem:otherK}.
For $i\in\{1,2\}$, we denote by $x_i, \sigma^\ins_i$ and $\sigma^\out_i$ the objects defined in \eqref{defx} and \eqref{def:in.out1} when $C^\ext=C^\ext_i$.

\bigskip

\noindent
The rest of the proof consists in merging $C_1^\ext$ and $C_2^\ext$ into a single unbounded cluster by resampling $B^{x_1}_{[\sigma_1^{\ins}, \sigma_1^{\out}]}$ and $B^{x_2}_{[\sigma_2^{\ins}, \sigma_2^{\out}]}$ with excursions that do intersect each other.\\
Thus, we require that a rerouting of the excursions does not disconnect them from their respective cluster, hence Step 3.
This task is easier when both excursions have time length deterministically bounded from below, hence Step 4.
Conditioned on both endpoints, $B^{x}_{[\sigma^{\ins},\sigma^{\out}]}$ is  a Brownian excursion, the law of which is not absolutely continuous with respect to that of Brownian motion.
As a consequence, we cannot directly use our knowledge on the intersection probabilities of two Brownian motions, hence Step 5.


\begin{figure}[htbp]
\begin{center}
\includegraphics[height=6cm]{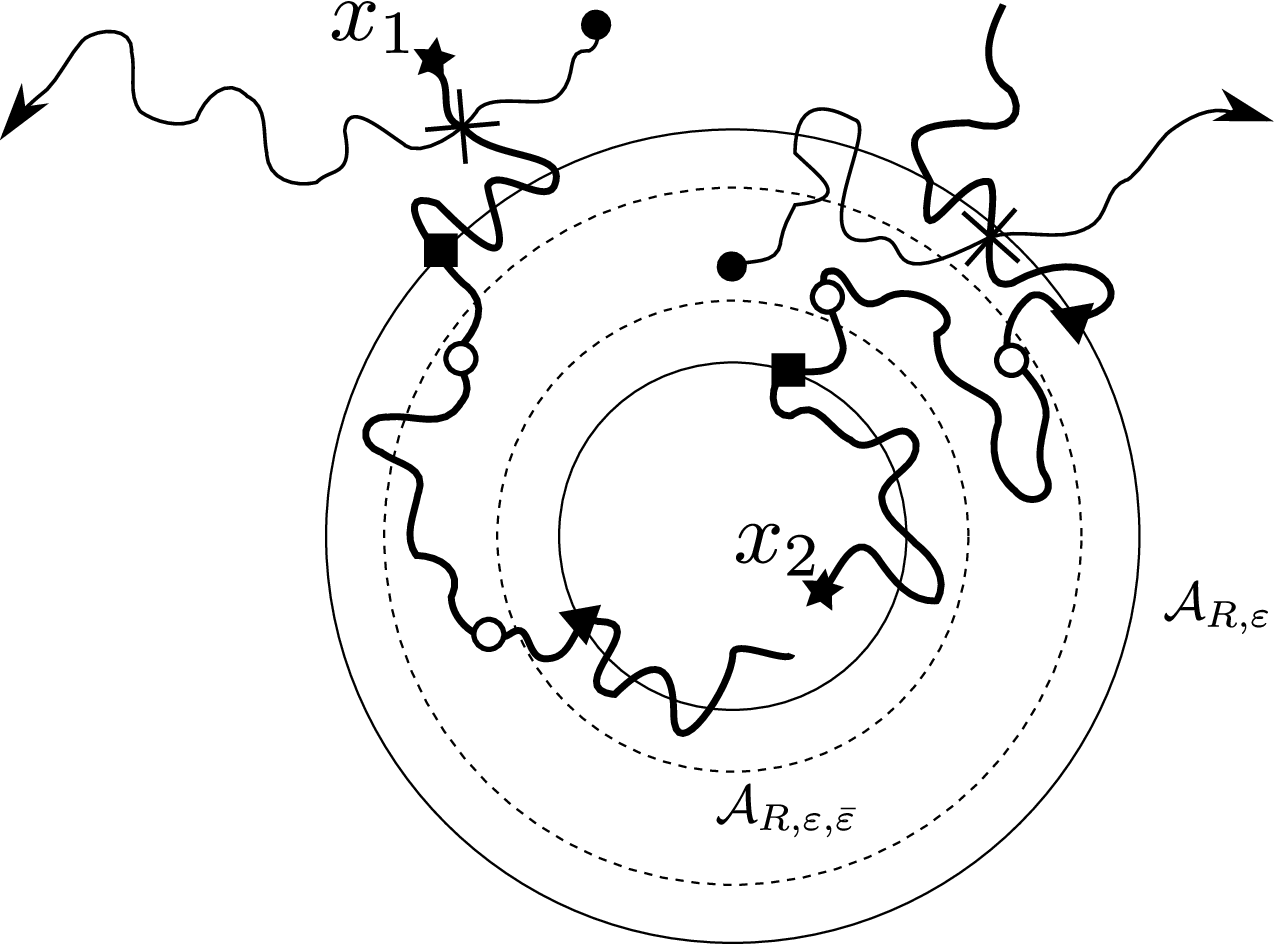} 
\end{center}
\caption{\small
In this picture the points marked with $\bigstar$ are $x_i, \ i=1,2$.
The symbols $\blacksquare,\blacktriangle$ refer to the times $\sigma^{\ins}$ and $\sigma^{\out}$, respectively.
The symbol $\circ$ represents the times $\sigma^{\ins}+\delta$ and $\sigma^{\out}-\delta$, respectively.
Finally, the symbol $\times$ indicates that Condition \eqref{econn} is fulfilled.
}
\label{annulus}
\end{figure}

\bigskip

\paragraph{Connecting $C_1$ and $C_2$ inside the annulus.}
The strategy announced above translates into the following lower bound for $\P(N_{\infty}=1)$:
\begin{equation}
\label{contra}
\begin{aligned}
&\P(\tilde{E}_{R,\epsilon,T,\bep,2}) \times\\
&\P\Big(
\Big\{B^{x_1}_{[\sigma_1^{\ins}+\delta,\sigma_1^{\out}-\delta]} \bigcap B^{x_2}_{[\sigma_2^{\ins}+\delta,\sigma_2^{\out}-\delta]}\neq \emptyset\Big\},
\bigcap\limits_{i=1,2}\Big\{B^{x_i}_{\sigma_i^{\ins}+2\delta},B^{x_i}_{\sigma_i^{\out}-2\delta}
\in \cAJ\Big\}
\ \Big| \ 
\tilde{E}_{R,\epsilon,T,\bep,2}
\Big).
\end{aligned}
\end{equation}
The reason for the $2\delta$ in \eqref{contra} is that the property mentioned in Step 5 only holds on time intervals which excludes neighbourhoods of the endpoints.

\noindent

\bigskip

\noindent {\bf Additional notation.}
At this point we would like to introduce some notations for ease of readability.\\
\noindent
First, let us introduce some events of interest. Let $s>r\geq 0$.
For a set $D\subset \R^d$, we denote by
\begin{equation}
\label{stay}
\pazocal{S}_{[r,s]}(D):=\{\Pi\in \mathcal{C}([0,\infty),\R^d):\, \Pi_{[r,s]} \subseteq D \},
\end{equation}
the set of all continuous paths which \emph{are contained in }$D$ during the time interval $[r,s]$, and by
\begin{equation}
\label{belong}
\pazocal{L}_{r,s}(D):=\{\Pi\in \mathcal{C}([0,\infty),\R^d):\,\Pi_r, \Pi_s \in D \},
\end{equation}
the set of all continuous paths which \emph{lie in the set} $D$ at times $r,s$.\\
\noindent
In the same fashion we also define for $s_1>r_1\geq 0$ and $s_2>r_2\geq 0$
\begin{equation}
\label{intersection}
\pazocal{I}_{[s_1,r_1],[s_2,r_2]}:=\Big\{\Pi^{(1)},\Pi^{(2)} \in \mathcal{C}([0,\infty),\R^d):\,\Pi_{[s_1,r_1]}^{(1)}\bigcap \Pi_{[s_2,r_2]}^{(2)}\neq \emptyset \Big\},
\end{equation}
the set of all pairs of continuous paths which, when restricted to the respective time intervals $[r_1,s_1]$ and $[r_2,s_2]$, have a non-empty intersection. \\
\noindent
Secondly, we slightly modify our previous notation: $\bbP^a_t$ now denotes the law of a Brownian motion starting at $a$ and running from time $0$ up to time $t$.
If we consider Brownian bridges instead of Brownian motions we substitute the letter $a$ by ${\bf a}=(\underline{a};\overline{a})$ containing the starting and ending positions of the Brownian bridge.
When considering two independent copies of a Brownian motion (resp.\ Brownian bridge) we add a superscript/subscript, i.e.\ $\bbP^{a_1,a_2}_{t_1,t_2}$ (resp.\ $\bbP^{\bf a_1,a_2}_{t_1,t_2}$).
Finally, we will refer to a Brownian bridge as $W$.

\bigskip

\noindent
{\bf Observation: }
For $i\in\{1,2\}$, conditionally on $T_i := \sigma_i^{\out} - \sigma_i^{\ins}$ and the endpoints $(B^{x_i}_{\sigma_i^{\ins}+\delta}, B^{x_i}_{\sigma_i^{\out}  -\delta}) = (a_i,b_i)$, $B^{x_i}_{[\sigma_i^{\ins}+\delta, \sigma_i^{\out}-\delta]}$ is a Brownian bridge running from $a_i$ to $b_i$ in a time interval of length $\tau_i := T_i -2\delta \geq \frac{3T}{4}$, conditioned to stay in $\cA_{R,\epsilon}$ 
(recall the definitions of $\sigma_i^{\ins}$ and $\sigma_i^{\out}$, $i\in\{1,2\}$).

\vspace{0.5cm}\noindent
The observation above together with \eqref{contra} yields
\begin{equation}
\label{eq:onecluster}
\begin{aligned}
\P&(N_\infty=1)\\
& \geq \P(\tilde{E}_{R,\epsilon,T,\bep,2})\, \inf_{\substack{\tau_1,\tau_2\geq 3T/4\\ {\bf a_1}, {\bf a_2} \in \cAJ^2}}
\bbP^{{\bf a_1},{\bf a_2}}_{\tau_1,\tau_2}
\Big(
\pazocal{L}^i_{\delta, \tau_i-\delta}(\cAJ) \ , \ \pazocal{S}^i_{[0,\tau_i]}(\pazocal{A}_{R,\varepsilon}
)
,\  1\leq i \leq 2, \quad \pazocal{I}_{[0,\tau_1],[0,\tau_2]}
\Big)
\end{aligned}
\end{equation}
and the superscript $i\in\{1,2\}$ refers to the $i$-th copy of the corresponding processes.
Since $\P(\tilde{E}_{R,\epsilon,T,\bep,2})>0$, by Steps 1--5, it is enough to prove that
\begin{equation}\label{eq:Pcap}
\inf_{\substack{\tau_1,\tau_2\geq 3T/4\\ {\bf a_1}, {\bf a_2} \in \cAJ^2}}
\bbP^{{\bf a_1},{\bf a_2}}_{\tau_1,\tau_2}
\Big(
\pazocal{L}^i_{\delta, \tau_i-\delta}(\cAJ) \ , \ \pazocal{S}^i_{[0,\tau_i]}(\pazocal{A}_{R,\varepsilon}
)
, \  1\leq i \leq 2, \quad \pazocal{I}_{[0,\tau_1],[0,\tau_2]}
\Big) > 0.
\end{equation}

\paragraph{Proof of Equation \eqref{eq:Pcap}.} 
We fix ${\bf a_1, a_2}\in\cAJ$ and $\tau_1,\tau_2\geq 3T/4$. The left-hand side of (\ref{eq:Pcap}) may be bounded from 
below by
\begin{equation}
\label{eq:Pcapbelow}
\bbP^{{\bf a_1},{\bf a_2}}_{\tau_1,\tau_2}
\Big(
\pazocal{L}^i_{\delta, \tau_i-\delta}(\cAJ) \ , \ \pazocal{S}^i_{[0,\tau_i]}(\pazocal{A}_{R,\varepsilon})
, \  1\leq i \leq 2, \quad \pazocal{I}_{[0,\tau_1-\delta],[0,\tau_2-\delta]}
\Big),
\end{equation}
which equals, by the Markov property applied at times $\tau_i-\delta$, $i\in\{1,2\}$,
\begin{equation}\label{eq:PcapMarkov}
\bbE^{{\bf a_1},{\bf a_2}}_{\tau_1,\tau_2}
\bigg(
\prod_{i=1,2} \one
\big\{
\pazocal{L}^i_{\delta, \tau_i-\delta}(\cAJ) \ , \ \pazocal{S}^i_{[0,\tau_i-\delta]}(\pazocal{A}_{R,\varepsilon})
\big\}
\ 
\one\big\{\pazocal{I}_{[0,\tau_1-\delta],[0,\tau_2-\delta]}
\big\}
\Phi_\delta(W^{(i)}_{\tau_i-\delta};\overline{a}_i)
\bigg)
\end{equation}
where
\begin{equation}
\Phi_\delta({\bf a}) := \bbP^{\bf a}_\delta(\pazocal{S}_{[0,\delta]}(\cA_{R,\epsilon})), \qquad {\bf a}=(\underline{a},\overline{a}) \in (\R^d)^2,
\end{equation}
is the probability that a Brownian bridge going from $\underline{a}$ to $\overline{a}$ within the time interval $[0,\delta]$ stays in $\cA_{R,\epsilon}$.
\noindent
To bound (\ref{eq:PcapMarkov}) from below we use the following three lemmas, whose proofs may be found in the appendix of \cite{EMP13}.
\begin{lemma}{\bf (Positive probability for a Brownian bridge to stay inside the annulus)}
\label{lem:stayinA}
There exists $c>0$ such that for all ${\bf a}\in\cAJ^2$, $\Phi_\delta({\bf a})\geq c$.
\end{lemma}

\begin{lemma}{\bf (Substitution of the Brownian bridge by a Brownian motion)]}
\label{lem:density}
Let $\tau > 0$ and $\delta\in (0,\tau)$. There exists $c>0$ such that for all ${\bf a} = (\underline{a},\overline{a}) \in\cAJ^2$,
\begin{equation}
\label{eq:density}
\frac{
\dd \bbP^{\bf a}_\tau
(
W_{[0,\tau-\delta]}\in \cdot \ , \ \pazocal{L}_{\delta,\tau-\delta}(\cAJ)
)
}{
\dd \bbP^{\underline{a}}_\tau
(
B_{[0,\tau-\delta]}\in \cdot \ , \ \pazocal{L}_{\delta,\tau-\delta}(\cAJ)
)
}
\geq c.
\end{equation}
\end{lemma}

\begin{lemma}{\bf (Two Brownian motions restricted to be inside the annulus do intersect)}
\label{lem:intersect}
Let $\tau_1, \tau_2>0$ and $0 < \delta < \frac{\tau_1 \wedge \tau_2}{2}$. There exists $c>0$ such that for all $a_1,a_2\in\cAJ$
\begin{equation}
\label{eq:intersect}
\bbP^{a_1,a_2}_{\tau_1,\tau_2}
\Big(
\pazocal{L}^i_{\delta, \tau_i-\delta}(\cAJ) \ , \ \pazocal{S}^i_{[0,\tau_i-\delta]}(\pazocal{A}_{R,\varepsilon})
,\  1\leq i \leq 2, \quad \pazocal{I}_{[0,\tau_1-\delta],[0,\tau_2-\delta]}
\Big)
\geq c.
\end{equation}
\end{lemma}

\vspace{0.5cm}\noindent
We now explain how to get (\ref{eq:Pcap}) by applying  Lemmas \ref{lem:stayinA}--\ref{lem:intersect} to \eqref{eq:PcapMarkov}.
Since the $W_{\tau_i-\delta}$'s, $i\in\{1,2\}$, appearing in \eqref{eq:PcapMarkov} are in $\cAJ$, Lemma \ref{lem:stayinA} yields that, for some $c>0$, \eqref{eq:PcapMarkov} is greater or equal to
 \begin{equation}
 \label{eq:applystayinA}
 c^2 \
 \bbP^{{\bf a_1},{\bf a_2}}_{\tau_1,\tau_2}
 \Big(
 \pazocal{L}^i_{\delta, \tau_i-\delta}(\cAJ) \ , \ \pazocal{S}^i_{[0,\tau_i-\delta]}(\pazocal{A}_{R,\varepsilon})
 ,\  1\leq i \leq 2, \quad \pazocal{I}_{[0,\tau_1-\delta],[0,\tau_2-\delta]}
 \Big).
 \end{equation}
Next, a change of measure argument together with the bound on the Radon-Nikodym derivative provided in Lemma
\ref{lem:density} yields, for a possibly different constant $c>0$, that (\ref{eq:applystayinA}) is at least
\begin{equation}
\label{eq:applydensity}
c \
 \bbP^{\underline{a_1},\underline{a_2}}_{\tau_1,\tau_2}
 \Big(
 \pazocal{L}^i_{\delta, \tau_i-\delta}(\cAJ) \ , \ \pazocal{S}^i_{[0,\tau_i-\delta]}(\pazocal{A}_{R,\varepsilon})
,\  1\leq i \leq 2, \quad \pazocal{I}_{[0,\tau_1-\delta],[0,\tau_2-\delta]}
 \Big) ,
\end{equation}
which is positive by Lemma \ref{lem:intersect}. To deduce (\ref{eq:Pcap}) from it, it is enough to note that all the previous estimates are uniform in ${\bf a_1,a_2}\in \cAJ$. 
This finally yields the claim.

\begin{remark}
\label{rem:otherK}
If $n_0>2$ in \eqref{positive}, then one follows the same scheme and ends up connecting more than two excursions in an annulus.
Using the same proof as for two excursions, one can connect $B^{x_1}_{[\sigma_1^\ins,\sigma_1^\out]}$ to $B^{x_i}_{[\sigma_i^\ins,\sigma_i^\out]}$ during the time interval $[\sigma_1^\ins + (i-1)\delta/n_0, \sigma_1^\ins + i\delta/n_0]$, where $\delta \in (0,T)$, for all $1\leq i\leq n_0$. The same argument applies when we assume $N_\infty = k >2$ a.s.
\end{remark}

\subsubsection{Excluding $N_\infty = \infty$}
\label{S5.3.2}

Let us assume that the number $N_{\infty}$ of unbounded clusters in $\pazocal{O}_t$ is almost-surely equal to infinity. 
In the same fashion as in Section \ref{S5.2.2} we show that this leads to a contradiction.
For $z\in\bbZ^d$, we define the event
\begin{equation}
\label{eq:EN}
E_R(2Rz) := 
\left\{
\parbox{11cm}
{
there exists an unbounded cluster $C$ such that 
$C\cap \cpl{\cB_{\infty}(2Rz,R)}$ contains at least three unbounded clusters
and any unbounded cluster which intersects $\cB_{\infty}(2Rz,R)$
equals $C$
}
\right\}.
\end{equation}
Note that for all $k\geq 3$,
\begin{equation}
\label{eq:writeallintersect}
\begin{aligned}
E_R(2Rz) 
&\supseteq 
\left\{
\begin{array}{cc}
\mbox{there exists } k \mbox{ unbounded clusters in $C\cap \cpl{\cB_{\infty}(2Rz,R)}$
}\\
\mbox{ and all of them are connected inside $\cB_{\infty}(2Rz,R)$}
\end{array}
\right\}.
\end{aligned}
\end{equation}
Hence, Remark \ref{rem:otherK} and a short decomposition argument yield that the last event in (\ref{eq:writeallintersect})
has positive probability for $R$ large enough. Consequently, so does $E_R(2Rz)$. From now on, the proof works
similarly as that of Section \ref{S5.2.2}. Thus, to avoid repetitions we just point out the differences with the proof in Section \ref{S5.2.2}.\\
\noindent
The identification done in {\bf STEP 2.} of Section \ref{S5.2.2} has to be changed.
For each $z\in\Z^d$, we replace the Poisson point inside $\pazocal{B}_\infty(2Rz,R)$ that was used to connect the ``external'' clusters by what we call an \emph{intersection point}. This point is just an arbitrarily chosen point $\tilde{z}\in \cB_{\infty}(2Rz,R)$
contained in all the clusters. The collection of such points $\tilde{z}$ constitute the set $\pazocal{R}$ in the present case.
Finally, at the moment of applying Lemma \ref{lem:counting}, we define
$$
C_z^{i} = 
\left\{
x \in \{\pazocal{E}\cap\cB_{\infty}(0,LR)\}\cup \{\text{intersection points}\} \colon \ \con{x}{\Lambda_z^c}{\mathbf{B}^i_z}
\right\},
\qquad 
i=1, \ldots, n_z
$$
and
$$
 \pazocal{S}=\pazocal{B}_\infty(0,L R) \cap (\pazocal{E} \cup \{\text{intersection points}\}).
$$
The contradiction is now obtained in a similar fashion as in \eqref{eq:contradiction}, subject to minor modifications. We omit the details.

\appendix
\section{Proof of Lemma \ref{lem:renormalization}}
\label{A1}
The proof consists of two steps. In the first step a coarse-graining procedure is introduced,
which reduces the problem of showing subcriticality of a continuous percolation model to showing subcriticality of an infinite range site percolation model
on $\Z^d$. This coarse-graining was essentially already introduced in \cite[Lemma 3.3]{MR96}, where $\rho$ was supposed to have a compact support.
To overcome the additional difficulties arising from the long range dependencies in the coarse-grained model, we use a renormalization scheme, which is 
similar to the
one in Sznitman \cite[Theorem 3.5]{S10}.
\medskip

\noindent
{\bf STEP 1. Coarse-graining.}\\
We fix $N\in\N$. For $n\in\N$, a sequence of vertices $z_0,z_1,\ldots, z_{n-1}$ in $\Z^d$ is called a $*$-path when $\|z_i-z_{i-1}\|_{\infty}=1$ for all
$i\in\{1,2,\ldots,n-1\}$.
Furthermore, a site $z = (z(j), 1\leq j \leq d)\in\Z^d$ is called open when there is an occupied cluster
$\Lambda$ of $\Sigma$ such that
\begin{equation}
\label{eq:opensite}
\begin{aligned}
\text{(i)}\ \Lambda\cap \prod_{j=1}^{d}[z(j)N, (z(j)+1)N)\neq \emptyset \mbox{   and   }
\text{(ii)}\ \Lambda\cap \cpl{\Bigg(\prod_{j=1}^{d}[(z(j)-1)N,(z(j)+2)N)\Bigg)}\neq
\emptyset.
\end{aligned}
\end{equation} 
Otherwise, $z$ is called closed. It was shown in \cite[Lemma 3.3]{MR96} that to obtain Lemma \ref{lem:renormalization} it suffices 
to show that
\begin{equation}
\label{eq:zerocluster}
\P_{\lambda,\rho}\Big(\mbox{$0$ is contained in an infinite $*$-path
of open sites}\Big)=0.
\end{equation}
To prove (\ref{eq:zerocluster}) we introduce a renormalization scheme.
\medskip

\noindent
{\bf STEP 2. Renormalization.}\\
$\bullet$ {\bf New notation and a first bound.} 
We start by introducing new notations.
We fix integers $R>1$ and $L_0>1$, both to be determined
and we introduce an increasing sequence of scales via 
\begin{equation}
L_{n+1}=R^{n+1}L_n,\quad n\in\N_0.
\end{equation}
Moreover, for $i\in\Z^d$, we introduce a sequence of increasing boxes via
\begin{equation}
\label{eq:renbox}
\begin{aligned}
& \Box_n(i) = \prod_{j=1}^{d}[i(j)L_n,(i(j)+1)L_n)\cap\Z^d\quad \mbox{and}\\
&\boxplus_n(i) = \prod_{j=1}^{d}[(i(j)-1)L_n,(i(j)+2)L_n)\cap\Z^d.
\end{aligned}
\end{equation}
We further abbreviate $\Box_n=\Box_n(0)$ and $\boxplus_n=\boxplus_n(0)$.
Thus, $\boxplus_n(i)$ is the union of boxes $\Box_n(j)$ such that $\|j-i\|_{\infty}\leq1$.
Moreover, for $n\in\N$, we introduce the events
\begin{equation}
\label{eq:cascade}
A_n(i)= \Big\{\mbox{there is a $*$-path of open sites from 
$\Box_n(i)$ to $\partial_{\mathrm{int}}\boxplus_n(i)$}\Big\} 
\end{equation}
and we write $A_n$ instead of $A_n(0)$.
Here, $\partial_{\mathrm{int}}\Delta$ refers to the inner boundary of a set $\Delta\subseteq\Z^d$
with respect to the $\|\cdot\|_{\infty}$-norm.
The idea of the renormalization scheme is to bound the probability of $A_{n+1}$
in terms of the probability of the intersection of events
$A_n(i)$ and $A_n(k)$, where $i\in\Z^d$ and $k\in\Z^d$ are thought to be far apart.
By the assumption on the radius distribution $\rho$, the events $A_n(i)$
and $A_n(k)$ can then be treated as being almost independent.
This will result in a recursion inequality which relates the probabilities of the events 
$A_n$, $n\in\N$, at different scales. For that, we fix $n\in\N$ and let
\begin{equation}
\label{eq_H1H2}
\begin{aligned}
&\cH_1=\Big\{i\in\Z^d\, :\, \Box_n(i)\subseteq \Box_{n+1},
\Box_n(i)\cap \partial_{\mathrm{int}} \Box_{n+1}\neq \emptyset\Big\}\quad \mbox{and}\\
&\cH_2=\Big\{k\in\Z^d\,:\, \Box_n(k)\cap\Big\{z\in\Z^d\,:\, \mathrm{dist}(z,\Box_{n+1})
=\frac{L_{n+1}}{2}\Big\}\neq\emptyset\Big\}.
\end{aligned}
\end{equation}
Here, $\mathrm{dist}(z,\Box_{n+1})$ denotes the distance of $z$ from the set $\Box_{n+1}$
with respect to the supremum norm. Note that here and in the rest of the proof, for notational convenience, we pretend that expressions like $L_{n+1}/2$ are integers. 
Observe that if $A_{n+1}$ occurs, then there are $i\in\cH_1$ and $k\in\cH_2$
such that both $A_{n}(i)$ and $A_{n}(k)$ occur.
Hence,
\begin{equation}
\label{eq:recurs}
\begin{aligned}
\P_{\lambda,\rho}(A_{n+1})
&\leq \sum_{i\in\cH_1,k\in\cH_2} \P_{\lambda,\rho}(A_{n}(i)\cap A_{n}(k))\\
&\leq c_1R^{2(d-1)(n+1)}\sup_{i\in\cH_1,k\in\cH_2}\P_{\lambda,\rho}(A_n(i)\cap A_n(k)),
\end{aligned}
\end{equation} 
where $c_1=c_1(d)>0$ is a constant which depends only on the dimension.
\medskip

\noindent
$\bullet${\bf Partition of $A_{n}(i)\cap A_n(k)$.}
We fix $i\in\cH_1$ and $k\in\cH_2$. Let $z\in\boxplus_n(i)$
and note that to decide if $z$ is open, it suffices to know the trace of
the Boolean percolation model on
\begin{equation}
\label{eq:trace}
\prod_{j=1}^{d}[(z(j)-1)N,(z(j)+2)N).
\end{equation}
In a similar fashion one sees that the area which determines $A_n(i)$ is given by
\begin{equation}
\label{eq:determinei}
\begin{aligned}
&\prod_{j=1}^{d}[((i(j)-1)L_n-1)N,((i(j)+2)L_n+2)N]\\
&\subseteq \prod_{j=1}^{d}[(i(j)-2)L_nN,(i(j)+3)L_nN)
\stackrel{\mathrm{def}}{=} \mathrm{DET}(\boxplus_n(i))
\end{aligned}
\end{equation}
and likewise for $A_n(k)$.
Here, we used that by our choice of $R$ and $L_0$ the relation $L_n\geq 2$ holds 
for all  $n\in\N$. We introduce 
\begin{equation}
\cD(x,r(x)) := \{\cB(x,r(x)) \cap  \mathrm{DET}(\boxplus_n(i))\neq \emptyset,\, \cB(x,r(x)) \cap  \mathrm{DET}(\boxplus_n(k))\neq
\emptyset \}
\end{equation}
and
\begin{equation}
\label{eq:crossball}
U_n(i,k) := \bigcup_{x\in\cE}\cD(x,r(x)),
\end{equation}
so that
\begin{equation}
\label{eq:condition}
\begin{aligned}
\P_{\lambda,\rho}(A_n(i)\cap A_n(k))
&= \P_{\lambda,\rho}(A_n(i)\cap A_n(k)\big|\cpl{U_n(i,k)})
\ \P_{\lambda,\rho}(\cpl{U_n(i,k)})\\
&\ +\P_{\lambda,\rho}(A_n(i)\cap A_n(k)\big|U_n(i,k))
\ \P_{\lambda,\rho}(U_n(i,k)).
\end{aligned}
\end{equation}
\medskip

\noindent
$\bullet${\bf Analysis of the first term on the right-hand side of (\ref{eq:condition}).}
We claim that under $\P_{\lambda,\rho}(\cdot\big|\cpl{U_n(i,k)})$ the 
events $A_n(i)$ and $A_n(k)$ are independent. To see that, note that the Poisson point process $\chi$ on $\R^d\times[0,\infty)$ with intensity measure $\nu=(\lambda\times\mathrm{Leb}_d)\otimes\rho$ (see Section \ref{S2.1}) is a 
Poisson point process under $\P_{\lambda,\rho}(\cdot | \cpl{U_n(i,k)})$ with intensity measure 
\begin{equation}
\label{eq:intensitycond}
\one\{\mbox{there is no } (x,r(x))\in\chi \mbox{ such that } \cD(x,r(x)) \mbox{ occurs}\}\times \nu.
\end{equation}
However, on $ \cpl{U_n(i,k)}$, the events $A_n(i)$ and $A_n(k)$ depend on disjoint subsets of $\R^d\times [0,\infty)$.
Consequently, they are independent under  $\P_{\lambda,\rho}(\cdot\big|\cpl{U_n(i,k)})$.
Hence,
\begin{equation}
\label{eq:condindep}
\begin{aligned}
\P_{\lambda,\rho}&(A_n(i)\cap A_n(k)\big|\cpl{U_n(i,k)})
\ \P_{\lambda,\rho}(\cpl{U_n(i,k)})\\
&=\P_{\lambda,\rho}(A_n(i)\big|\cpl{U_n(i,k)})\ \P_{\lambda,\rho}(A_n(k)\big|\cpl{U_n(i,k)})\ \P_{\lambda,\rho}(\cpl{U_n(i,k)})\\
&\leq \P_{\lambda,\rho}(A_n)^2\ \P_{\lambda,\rho}(\cpl{U_n(i,k)})^{-1}.
\end{aligned}
\end{equation}
For the last inequality in (\ref{eq:condindep}) we used the fact that
$\P_{\lambda,\rho}(A_n(i))$ does not depend on $i\in\Z^d$.
\medskip

\noindent
$\bullet${\bf Analysis of the second term on the right-hand side of (\ref{eq:condition}).}
To bound the second term on the right-hand side of (\ref{eq:condition})
it will be enough to bound $\P_{\lambda,\rho}(U_n(i,k))$ from above, since the other
term is less than one.
Note that
\begin{equation}
\label{eq:ballcross}
\begin{aligned}
\P_{\lambda,\rho}(U_n(i,k))
\leq \sum_{\ell\in 3\Z^d} \P_{\lambda,\rho}
\Bigg(
\begin{array}{cl}
\exists x\in \cE \cap N\boxplus_{n+1}(\ell) :& \cB(x,r(x)) \cap \mathrm{DET}(\boxplus_n(i)) \neq \emptyset\\ 
\mbox{and } &\cB(x,r(x)) \cap \mathrm{DET}(\boxplus_n(k)) \neq \emptyset
\end{array}
\Bigg).
\end{aligned}
\end{equation}
Here, the set $N\boxplus_{n+1}(\ell)$ is the set $\{x\in\R^d\,:\, x=zN, z\in\boxplus_{n+1}(\ell)\}$.
We first treat the term $\ell=0$ in the sum in (\ref{eq:ballcross}).
Note that for all $n\in\N$,
\begin{equation}
\label{eq:dist}
\mathrm{dist}(\mathrm{DET}(\boxplus_n(i)),
\mathrm{DET}(\boxplus_n(k))\geq \Big(\frac{L_{n+1}}{2}-8L_n\Big)N
\geq \frac{L_{n+1}}{3}N,
\end{equation}
provided $R$ and $L_0$ are chosen accordingly.
Thus, if there is a Poisson point whose corresponding ball intersects $\mathrm{DET}(\boxplus_n(i))$ and $\mathrm{DET}(\boxplus_n(k))$,
then its radius is at least $L_{n+1}N/6$.
This yields
\begin{equation}
\label{eq:l0}
\begin{aligned}
\P_{\lambda,\rho}&\Bigg(\begin{array}{cl}
\exists x\in \cE \cap N\boxplus_{n+1}:& \cB(x,r(x)) \cap \mathrm{DET}(\boxplus_n(i)) \neq \emptyset\\ 
\mbox{and } &\cB(x,r(x)) \cap \mathrm{DET}(\boxplus_n(k)) \neq \emptyset
\end{array}\Big)\\
&\leq \P_{\lambda,\rho}\Big(\exists x\in \cE \cap N\boxplus_{n+1}: r(x) \geq L_{n+1}N/6 \Bigg).
\end{aligned}
\end{equation}
We may bound the right-hand side of (\ref{eq:l0}) by
\begin{equation}
\label{eq:calculate}
1-\exp\big\{-\lambda\mathrm{Leb}_d(N\boxplus_{n+1})
\rho([L_{n+1}N/6,\infty))\big\},
\end{equation}
which is at most $\lambda \mathrm{Leb}_d(N\boxplus_{n+1})
\rho([L_{n+1}N/6,\infty))$.
By our assumption on the radius distribution, for $R$ and $L_0$ large enough,
there is a constant $c_2=c_2(\rho)>0$ such that
the last term may be bounded from above by
$\lambda(3L_{n+1}N)^de^{-c_2L_{n+1}N/6}$.
The case $\ell > 0$ is treated in a similar manner.
Thus, the left-hand side of (\ref{eq:ballcross})
is at most
\begin{equation}
\label{eq:upbound}
\lambda (3L_{n+1}N)^de^{-c_2L_{n+1}N/6}
+ \sum_{m=1}^{\infty}\sumtwo{\ell\in 3\bbZ^d}{\|\ell\|_\infty = m}\lambda (3L_{n+1}N)^d
\times e^{-c_2(3(m-1)+1/2)L_{n+1}N}.
\end{equation}
This is bounded from above by
\begin{equation}
\label{eqfinbound}
c_3\lambda (3L_{n+1}N)^de^{-c_2L_{n+1}N/6}
\end{equation}
for some constant $c_3>0$ which is independent of $R$, $L_0$ and $N$.
Hence, we have bounded the second term on the right-hand side of (\ref{eq:condition}).
In particular, by the above considerations, we deduce that for all $n\in\N$ and for a suitable choice
of $R$ and $L_0$, $\P_{\lambda,\rho}(\cpl{U_n(i,k)})\geq 1/2$.
\medskip

\noindent
$\bullet${\bf Analysis of the recursion scheme.}
Equation (\ref{eq:recurs}) in combination with (\ref{eq:condition}) 
and the arguments following it show that
\begin{equation}
\label{eq:Anbound}
\P_{\lambda,\rho}(A_{n+1})
\leq 2c_1R^{2(d-1)(n+1)}\Big(\P_{\lambda,\rho}(A_n)^2 
+ c_3\lambda (3L_{n+1}N)^de^{-c_2L_{n+1}N/6}\Big).
\end{equation}
To proceed, we put 
\begin{equation}
a_n= 2c_1R^{2(d-1)n}\P_{\lambda,\rho}(A_n),\quad n\in\bbN.
\end{equation}

\begin{claim}
\label{cl:recurs}
For $R$ large enough, for all $n\in \bbN$ and for all $L_0\geq 2R^{4(d-1)+1}$, the inequality $a_n\leq L_n^{-1}$ implies that
$a_{n+1}\leq L_{n+1}^{-1}$.
\end{claim}
\begin{proof}
Let $n\in\N$ and assume that $a_n\leq L_n^{-1}$. Then,
\begin{equation}
\label{eq:an}
\begin{aligned}
a_{n+1}
&=2c_1R^{2(d-1)(n+1)} \P_{\lambda,\rho}(A_{n+1})\\
&\leq 4c_1^2R^{4(d-1)(n+1)}\Big[\P_{\lambda,\rho}(A_n)^2 + c_3\lambda
(3L_{n+1}N)^de^{-c_2L_{n+1}N/6}\Big]\\
&= a_n^2R^{4(d-1)} +4c_1^2c_3R^{4(d-1)(n+1)}\lambda(3L_{n+1}N)^de^{-c_2L_{n+1}N/6}.
\end{aligned}
\end{equation}
Thus, it is enough to show that
\begin{equation}
\label{eq:needtoshow}
a_n^2R^{4(d-1)}\leq (2L_{n+1})^{-1}
\quad\mbox{and}\quad
4c_1^2c_3R^{4(d-1)(n+1)}(3L_{n+1}N)^de^{-c_2L_{n+1}N/6}
\leq (2L_{n+1})^{-1}.
\end{equation}
For that, note that by our assumption on $a_n$,
\begin{equation}
\label{eq:1term}
a_n^2R^{4(d-1)}2L_{n+1}
\leq 2L_n^{-2}R^{4(d-1)}L_{n+1}
=2R^{4(d-1)}\frac{R^{n+1}}{R^nL_{n-1}}
\leq 2R^{4(d-1)+1}L_0^{-1}.
\end{equation}
Thus, choosing $L_0\geq 2R^{4(d-1)+1}$ yields the first desired inequality.
The second term on the right-hand side of (\ref{eq:an}) may be bounded from above using
similar considerations. This yields Claim \ref{cl:recurs}.
\end{proof}
\noindent
Hence, to use the claim, we need that $\P_{\lambda,\rho}(A_0)\leq L_0^{-1}$.
Observe that
\begin{equation}
\label{eq:A0}
\begin{aligned}
\P_{\lambda,\rho}(A_0)
&= \P_{\lambda,\rho}\Big(
\mbox{there is a $*$-path of open sites from $[0,L_0)^d$
to $\partial_{\mathrm{int}}[-L_0,2L_0)^d$}\Big)\\
&\leq \P_{\lambda,\rho}\Big(
\mbox{there is $z\in\partial_{\mathrm{int}}[-L_0,2L_0)^d$, which is open}\Big)\\
&\leq c_4L_0^{d-1} \P_{\lambda,\rho}(\mbox{$0$ is open}),
\end{aligned}
\end{equation}
where $c_4=c_4(d)>0$ does only depend on the dimension.
Equation (3.64) in \cite{MR96}
shows that 
\begin{equation}
\label{eq:opencross}
\P_{\lambda,\rho}(\mbox{$0$ is open})
\leq 2d\P_{\lambda,\rho}(\mathrm{CROSS}(N,3N,\ldots, 3N;\bbR^d)).
\end{equation}
Therefore, if the right-hand side of (\ref{eq:opencross}) is smaller than $(4dc_1c_4L_0^{d})^{-1}$,
we get from (\ref{eq:A0}) that $\P_{\lambda,\rho}(A_0)\leq (2c_1L_0)^{-1}$, thus $a_0\leq L_0^{-1}$.
Note that an infinite $*$-path of open sites containing zero implies $A_n$ for all $n\in\bbN$. Thus, Claime \ref{cl:recurs} finally yields
\begin{equation}
\label{eq:nopath}
\P_{\lambda,\rho}\Big(\mbox{$0$ is contained in an infinite $*$-path
of open sites}\Big)
\leq \lim_{n\to\infty}\P_{\lambda,\rho}(A_n) = 0.
\end{equation}
Consequently, Lemma \ref{lem:renormalization} holds for $\varepsilon\leq (4dc_1c_4L_0^{d+1})^{-1}$.


\def\cprime{$'$}


\begin{thebibliography}{MMS86}

\bibitem[vdBMW97]{vdBMW97}
J.\ van den Berg and R.\ Meester and D.\ White.
\newblock Dynamic Boolean models.
\newblock {\em  Stochastic Process. Appl.} 69(2):247--257, 1997.

\bibitem[BK89]{BK89}
R.~M. Burton and M.~Keane.
\newblock Density and uniqueness in percolation.
\newblock {\em Comm. Math. Phys.}, 121(3):501--505, 1989.

\bibitem[CFS08]{CFS08}
R.\ \u{C}ern\'y and S.\ Funken and E.\ Spodarev.
\newblock On the Boolean Model of Wiener Sausages.
\newblock {\em Methodol.\ Comput.\ Appl.\ Probab.}, 10:10--23, 2008.

\bibitem[EMP13]{EMP13}
D.\ Erhard and J.\ Martínez and J.\ Poisat.
Brownian Paths Homogeneously Distributed in Space: Percolation Phase Transition and Uniqueness of the Unbounded Cluster.
ArXiv:1311.2907v1, 2013.

\bibitem[EP15]{EP15}
D.\ Erhard and J.\ Poisat
Asymptotics of the critical time in Wiener sausage percolation with a small radius.
ArXiv:1503.01712, 2015.

\bibitem[GKN92]{GKN92}
A.~Gandolfi and M.~S. Keane and C.~M. Newman.
\newblock Uniqueness of the infinite component in a random graph with
  applications to percolation and spin glasses.
\newblock {\em Probab. Theory Related Fields}, 92(4):511--527, 1992.

\bibitem[G61]{G61}
E.~N.~Gilbert.
\newblock Random Plane Networks.
\newblock{\em Journal of the Society for Industrial and Applied Mathematics,}
9(4): 533--543, 1961.

\bibitem[Gou08]{G08}
J.-B.~Gou{\'e}r{\'e}.
\newblock Subcritical regimes in the {P}oisson {B}oolean model of continuum
  percolation.
\newblock {\em Ann. Probab.}, 36(4):1209--1220, 2008.

\bibitem[Gri00]{G00}
G.~R.~Grimmett.
\newblock Percolation.
\newblock In {\em Development of mathematics 1950--2000}, pages 547--575.
  Birkh\"auser, Basel, 2000.

\bibitem[KS91]{KS91}
I.~Karatzas and S.~E.~Shreve.
\newblock {\em Brownian motion and stochastic calculus}, volume 113 of {\em
  Graduate Texts in Mathematics}.
\newblock Springer-Verlag, New York, second edition, 1991.

\bibitem[KKP05]{KKP05}
G.\ Kesidis and T.\ Konstantopoulos and S.\ Phoha.
\newblock Surveillance coverage of sensor networks under a random
mobility strategy.
\newblock In {\em IEEE Sensors Conference}, Toronto, October, 2003. Proceedings
paper.

\bibitem[LSS97]{LSS97}
T.~M.~Liggett and R.~H.~Schonmann and A.~M.~Stacey.
\newblock Domination by product measures.
\newblock {\em Ann. Probab.}, 25(1):71--95, 1997.

\bibitem[MMS88]{MMS88}
M.~V. Menshikov and S.~A. Molchanov and A.~F. Sidorenko.
\newblock Percolation theory and some applications.
\newblock In {\em Probability theory. {M}athematical statistics. {T}heoretical
  cybernetics, {V}ol. 24 ({R}ussian)}, Itogi Nauki i Tekhniki, pages 53--110,
  i. Akad. Nauk SSSR Vsesoyuz. Inst. Nauchn. i Tekhn. Inform., Moscow, 1986.
\newblock Translated in J. Soviet Math. {{\bf{4}}2} (1988), no. 4, 1766--1810.

\bibitem[MP10]{MP10}
P.~M{\"o}rters and Y.~ Peres.
\newblock {\em Brownian motion}.
\newblock Cambridge Series in Statistical and Probabilistic Mathematics.
  Cambridge University Press, Cambridge, 2010.


\bibitem[MR94]{MR94}
R.~Meester and R.~Roy.
\newblock Uniqueness of unbounded occupied and vacant components in {B}oolean
  models.
\newblock {\em Ann. Appl. Probab.}, 4(3):933--951, 1994.

\bibitem[MR96]{MR96}
R.~Meester and R.~Roy.
\newblock {\em Continuum percolation}, volume 119 of {\em Cambridge Tracts in
  Mathematics}.
\newblock Cambridge University Press, Cambridge, 1996.

\bibitem[MRS94]{MRS94}
R.~Meester and R.~Roy and A.~Sarkar.
\newblock Nonuniversality and continuity of the critical covered volume
  fraction in continuum percolation.
\newblock {\em J. Statist. Phys.}, 75(1-2):123--134, 1994.

\bibitem[Pen95]{P95}
M.~D.~Penrose.
\newblock Continuity of critical density in a Boolean model.
\newblock {\em \emph{Unpublished notes}}, 1995.

\bibitem[PSSS13]{PSSS13}
Y.\ Peres and A.\ Sinclair and P.\ Sousi and A.\ Stauffer.
\newblock Mobile geometric graphs: detection, coverage and percolation.
\newblock {\em  Probab. Theory Related Fields}, 156(1-2):273--305, 2013.

\bibitem[PSS13]{PSS13}
Y.\ Peres and P.\ Sousi and A.\ Stauffer.
\newblock The isolation time of Poisson Brownian motions.
\newblock {\em  ALEA Lat. Am. J. Probab. Math. Stat.}, 10(2):813--829, 2013. 

\bibitem[Szn10]{S10}
A.-S.~Sznitman.
\newblock Vacant set of random interlacements and percolation.
\newblock {\em Ann. of Math. (2)}, 171(3):2039--2087, 2010.

\end{thebibliography}
\end{document}